\documentclass[a4paper,12pt]{amsart}
\usepackage{amsfonts,amsthm,amsmath,amssymb,enumitem,etoolbox,comment,graphicx,array,tabu,commath,latexsym,environ,hyperref,xcolor,tikz-cd,stackengine,diagbox,appendix,caption}
\usepackage[margin=1.2in]{geometry}
\newcommand{\RP}{{\mathbb{RP}}}
\newcommand{\R}{\mathbb R}

\newcommand{\Z}{\mathbb Z}
\newcommand{\Zc}{\mathcal Z}

\newcommand{\Fc}{\mathcal F}
\newcommand{\Pc}{\mathcal P}

\newcommand{\Fb}{\mathbf F}
\newcommand{\Mb}{\mathbf M}
\newcommand{\id}{\mathrm {id}}
\newcommand{\genus}{\mathrm {genus}}

\newcommand{\Hom}{\mathrm {Hom}}
\newcommand{\LSS}{\mathbf L_{\mathrm{SS}}}

\newcommand{\length}{\mathrm{length}}
\newcommand{\area}{\mathrm{area}}
\newcommand{\Hess}{\mathrm{Hess}}
\newcommand{\dist}{\mathrm{dist}}
\newcommand{\dmn}{\mathrm{dmn}}
\newcommand{\ind}{\mathrm{index}}

\renewcommand{\div}{\mathrm{div}}
\renewcommand{\tilde}{\widetilde}

\newcommand{\spt}{\mathrm{spt}}
\newcommand{\x}{\times}

\newtheorem{thm}{Theorem}[section]
\newtheorem{cor}[thm]{Corollary}
\newtheorem{prop}[thm]{Proposition}
\newtheorem{lem}[thm]{Lemma}

\theoremstyle{definition}
\newtheorem{defn}[thm]{Definition}

\newtheorem{rmk}[thm]{Remark}

\numberwithin{equation}{section}

\newtheoremstyle{TheoremNum}
{\topsep}{\topsep}{\itshape}{}{\bfseries}{.\,}{ }{\thmname{#1}\thmnote{ \bfseries #3}}
\theoremstyle{TheoremNum}

\makeatletter
\NewEnviron{restatethisthm}[1]{%
  \protected@write\@auxout{}{%
    \string\@restatetheorem{#1}{\detokenize\expandafter{\BODY}}%
  }%
  \begin{thm}\label{#1}\BODY\end{thm}%
}
\newcommand{\@restatetheorem}[2]{%
  \expandafter\gdef\csname restatethisthm@#1\endcsname{#2}%
}
\newcommand{\restatethm}[1]{%
  \begingroup
  \renewcommand{\thethm}{\ref{#1}}%
  \begin{thm}\csname restatethisthm@#1\endcsname\end{thm}%
  \endgroup
}

\NewEnviron{restatethisprop}[1]{%
  \protected@write\@auxout{}{%
    \string\@restateproposition{#1}{\detokenize\expandafter{\BODY}}%
  }%
  \begin{prop}\label{#1}\BODY\end{prop}%
}
\newcommand{\@restateproposition}[2]{%
  \expandafter\gdef\csname restatethisprop@#1\endcsname{#2}%
}
\newcommand{\restateprop}[1]{%
  \begingroup
  \renewcommand{\thethm}{\ref{#1}}%
  \begin{prop}\csname restatethisprop@#1\endcsname\end{prop}%
  \endgroup
}

\makeatother

\title{A  Free boundary minimal  surface via a 6-sweepout}
\date{\today}
\author{Adrian Chun-Pong Chu}
\address{Department of Mathematics, University of Chicago, 5734 S. University Avenue, Chicago, IL 60637, USA}
\email{acpc@uchicago.edu}

\begin{document}

\maketitle
\begin{abstract} We  prove that the Almgren-Pitts 6-width of the unit 3-ball is less than $2\pi$. We  also prove that there exists a free boundary minimal surface in the unit 3-ball that has genus at most 1, index at most 5,  area less than $2\pi$, and is not the equatorial disk or the critical catenoid.
\end{abstract}

\section{Introduction}\label{intro}

Given a compact Riemannian 3-manifold $M$, one can  {\it relate the  topology of the space of all surfaces in $M$ to minimal surfaces in $M$ via Morse theory}. One may even obtain information about the genus, Morse index, and  area of the minimal surfaces. In this paper, we will illustrate this phenomenon by looking at surfaces with low genus and area in the  compact Euclidean unit 3-ball $\mathbb B^3$, via the Almgren-Pitts and the Simon-Smith min-max theory.

Let $\mathcal E$ denote the set of all  surfaces, possibly with boundary, in $\mathbb B^3$ that are  smooth and properly embedded {\it except possibly at  finitely many points} (see \S \ref{trivial} for details). The reason for allowing singularities is that we want to study the space of all surfaces, regardless of their genus or number of connected components, as a whole. In fact, let us define on $\mathcal E$ the following topology inspired by the Simon-Smith min-max theory: For each  finite set $P\subset\mathbb B^3$, we define on the subset 
\begin{equation}\label{singsub}
    \{S\in\mathcal E:S \backslash P\textrm{ is smooth and properly embedded}\}
\end{equation}
the topology induced by the graphical $C^\infty$-convergence within open sets  $U\subset\subset\mathbb B^3\backslash P$ (meaning  $\overline{U}\subset \mathbb B^3\backslash P$). Now, we collect all open sets in  (\ref{singsub}) for all possible $P$ to form a base, thereby defining a topology on $\mathcal E$. Note that under this topology, one has continuous paths in $\mathcal E$ of surfaces with different genus or number of connected components via neck-pinching. Then our first main result is the following.

Let   $\mathcal E_g\subset \mathcal E$ be the subset of  {\it smooth} surfaces   with genus $g$,   and $\mathcal E^{a}\subset \mathcal E$ the subset of  surfaces with area less than $a$ for $a\in(0,\infty]$. Note that $\mathcal E^\infty\ne\mathcal E$, since an element of $\mathcal E$ can have an infinite area,  concentrated near a singularity. And as in the Simon-Smith min-max theory,  the genus of a disconnected smooth surface is defined as the sum of the genus of each of its connected components.

\begin{thm}\label{type}
The first to the sixth cohomology groups  of $$\overline{\mathcal E_0\cup \mathcal E_1}\cap \mathcal E^{2\pi}$$ in $\Z_2$-coefficients  are non-trivial: In fact, the  cup-length of this space  is at least $6$. And the same is true for any subspace of $\mathcal E^{\infty}$ that contains  $\overline{\mathcal E_0\cup \mathcal E_1}\cap\mathcal E^{2\pi}$.

\end{thm}

Note that $\overline{\mathcal E_0\cup \mathcal E_1}$ denotes the closure of $\mathcal E_0\cup \mathcal E_1$ in $\mathcal E$,  and the {\it cup-length} of a space $X$ is defined as the maximum number of elements in the cohomology ring of $X$ with degree at least 1 such that their cup product is non-trivial. We remark that $2\pi$ is twice the area of the equatorial disk in $\mathbb B^3$.

Let us mention the following results. In his celebrated work \cite{Hat83}, Hatcher proved the Smale conjecture, implying that the space  of smoothly  embedded  2-spheres in the (round) 3-sphere deformation retracts to the subspace of great 2-spheres, which is homeomorphic to $\RP^3$ and thus has cup-length 3. Moreover, based on Marques-Neves' ground-breaking resolution of the Willmore conjecture \cite{MN14}, Nurser showed that the space of {\it flat 2-cycles} in the unit round 3-sphere with area at most $2\pi^2$ (which is the area of the Clifford torus) has cup-length in $\Z_2$-coefficients at least 7 \cite{Nur16}.

Theorem \ref{type}   follows immediately from the  result below, of which the terminologies will be defined precisely in \S \ref{minmax}. 

\vbox{
\begin{thm}\label{6sweep}
There exists in the Euclidean unit $3$-ball a  family $\Psi$ of surfaces such that:
\begin{enumerate}[label=(\Alph*)]
    \item \label{cond5sweep} $\Psi$ is a $6$-sweepout in the sense of Almgren-Pitts min-max theory.
    \item \label{condgenusbdd} $\Psi$ is a smooth family of surfaces with genus at most $1$, in the sense of  Simon-Smith min-max theory.
    \item \label{condareabdd} The  area of each element in $\Psi$ is less than $2\pi$.
\end{enumerate}
\end{thm}
}

Theorem \ref{6sweep} also gives the following result immediately.

\begin{cor}\label{6width}
The Almgren-Pitts $6$-width of the Euclidean unit $3$-ball is less than $2\pi$. 
\end{cor}

Currently, the Almgren-Pitts widths of $\mathbb B^3$ are not well-understood: While the first three widths are $\pi$ and are detected by the equatorial disk (since the collection of flat disks in $\mathbb B^3$ is a 3-sweepout), the fourth already seems to be unknown.
Regarding computations of  Almgren-Pitts widths of other manifolds, see also \cite{Aie19,BL22,CM21,Don22,Zhu22}.

Let us now turn to the other side of the story: Free boundary minimal surfaces in the $\mathbb B^3$. In recent years, besides the two most basic examples, the equatorial disk and the critical catenoid, an abundance of free boundary minimal surfaces in $\mathbb B^3$ were constructed. For example, by solving extremal eigenvalue problems, Fraser-Schoen constructed examples with genus 0 and arbitrary number of boundary components \cite{FS16}. Using gluing techniques, Kapouleas-Li constructed embedded free boundary minimal surfaces of large genus that desingularize the union of the equatorial disk and the critical catenoid \cite{KL17}. (See  \cite{CFS20,CSW22,FPZ17,KM20,KZ21,Ket16,Ket16b,KW17} for more examples.)

We will use  min-max theory to produce a free boundary minimal surface. The advantage of this approach is that one can upper bound the Morse index of the minimal surface because of the work of Marques-Neves \cite{MN16}. In general, Morse index is difficult to compute. For example, to our best knowledge, in $\mathbb B^3$ the only embedded free boundary minimal surfaces  whose index are known are the equatorial disk and the critical catenoid: They have index $1$ and $4$ respectively \cite{Dev19,SZ19,Tra20}. In addition, from the recent resolution of the multiplicity one conjecture in the free boundary setting by Sun-Wang-Zhou \cite{SWZ20} based on the work of Zhou  \cite{Zho20}, we know there exists a sequence $\{\Sigma_k\}$ of embedded free boundary minimal surfaces in $\mathbb B^3$ with area growth of order $k^{1/3}$ and index at most $k$. However, using the Almgren-Pitts min-max theory, one cannot control the genus of the surfaces. In this paper, we apply the Simon-Smith min-max theory to the family $\Psi$ in Theorem  \ref{6sweep} to construct an  example with index, genus, and area bound:

\begin{thm}\label{min1} There exists in the Euclidean unit $3$-ball an embedded free boundary minimal surface  with genus $0$ or $1$, Morse index $4$ or $5$, and area in the range $(\pi,2\pi)$, that is not the equatorial disk or the critical catenoid.
\end{thm}

In fact, using the results of Sargent \cite{Sar17} and Ambrozio-Carlotto-Sharp \cite{ACS18} that lower bound the index of a free boundary minimal surface by its genus and number of boundary components, we know that the surface in Theorem \ref{min1} has at most 16 boundary components. But we believe this bound is far from optimal (see \S \ref{openq} below). We also note that, since we have to prove the index bound,  we cannot use the equivariant min-max theory of Ketover \cite{Ket16}. 

\begin{rmk} \label{cfs}
We remark that  Carlotto-Franz-Schulz \cite{CFS20} showed,  using equivariant min-max theory, there exists a free boundary minimal surface in $\mathbb B^3$
that has genus 1, area less than $3\pi$, a connected boundary,   and symmetry group $D_2$, where $D_2\subset SO(3)$ denotes the dihedral group with four elements (see the Geometric Analysis Gallery by Schulz \cite{Sch}).
In fact, as we will see in \S \ref{cfsproof}, Theorem \ref{6sweep} can reproduce their result and slightly improve the area bound from $3\pi$ to $2\pi$.
\end{rmk}

The family $\Psi$ in Theorem \ref{6sweep} can be modified to become a desirable 6-sweepout in $\R^3$ equipped with the Gaussian metric $\frac 1{4\pi} e^{-|{\bf x}|^2/4}g_0$, in which $g_0$ denotes the Euclidean metric, allowing one to construct a  self-shrinker with genus, index and Gaussian area control. However,  the Gaussian metric has a singularity at infinity, which  poses some challenges in carrying out the min-max theory.  We plan to address this in our upcoming work.

\subsection{Open questions.} \label{openq}
Regarding Theorem \ref{type}, it would be interesting to change the genus 0 and 1 constraint, the area bound  $2\pi$, or the ambient space $\mathbb B^3$ (to  the round 3-sphere for example), and investigate the topology of the corresponding space of surfaces. It will be nice to have more examples of  $k$-sweepouts for $k>6$ that are  smooth families. And for each $k$, among $k$-sweepouts $\Phi$ that are  smooth families, is there a non-trivial lower bound for the maximum of the genus of elements in $\Phi$?

We conjecture that the free boundary minimal surface in Theorem \ref{min1}, denoted $\Sigma$, has index 5. One can also ask if $\Sigma$ has the third lowest  area among all free boundary minimal surfaces in $\mathbb B^3$, after the equatorial plane and the critical catenoid. Moreover, we speculate that $\Sigma$ is the same as the free boundary minimal surface constructed by Carlotto-Franz-Schulz \cite{CFS20} mentioned  in Remark \ref{cfs}.

Concerning the Almgren-Pitts min-max theory in $\mathbb B^3$, we conjecture the  4-width is detected by the critical catenoid $\mathbb K$. In particular, showing the 4-width is at least $\area(\mathbb K)$ seems challenging, as it may depend on the conjecture that the second least area of an immersed free boundary minimal surface in $\mathbb B^3$
is realized by
the critical catenoid \cite[\S 7]{Li19}. As for the 5-width and the 6-width, it will be interesting to know if they are detected by the free boundary minimal surface  of Theorem \ref{min1}.

\subsection{Overview of  proofs}
Let us outline the construction of the smooth family $\Psi$ in Theorem \ref{6sweep}. We first consider the saddle surface $\{x^2-y^2+z=0\}$ in $\R^3$, and then translate, rescale, and rotate it arbitrarily: We even allow the scaling factor to be $0$ or $\pm\infty$. Then we collect all such surfaces, and it turns out this collection can be  parametrized by a  7-dimensional  quotient space of some $D_2$-action on $\RP^4\x SO(3)$. This is actually due to the $D_2$-symmetry of the saddle. However, this collection contains intersecting planes like $\{x^2-y^2=0\}$, the blow down of the saddle, which has a singular line and thus is not allowed in the Simon-Smith setting. To resolve this, we {\it desingularize the intersecting planes by adding a small $z^3$ term to their defining equations} (e.g. see Figure \ref{fig:phi1} and Table \ref{table::openup}), so that only isolated singularities appear. Finally, we intersect all surfaces with $\mathbb B^3$ to define $\Psi$.

 Theorem \ref{type} is an immediate consequence of Theorem \ref{6sweep}: See \S \ref{typeproof}.

Finally, for Theorem \ref{min1}, we will use the smooth family $\Psi$ in Theorem \ref{6sweep} as follows. Let $\Psi^{(5)}$ denote the subfamily of $\Psi$ parametrized by a 5-skeleton of the parameter space of $\Psi$. By applying the Simon-Smith min-max theorem to $\Psi^{(5)}$, we obtain a free boundary minimal surface $\Gamma$ with genus at most 1, index at most 5, and area less than $2\pi$. Note that, although it is not known if the multiplicity one conjecture holds in the Simon-Smith setting, we can  guarantee that $\Gamma$  has multiplicity one because $\area(\Gamma)<2\pi$ and the least possible area of a free boundary minimal surface in $\mathbb B^3$ is $\pi$. Then by the fact that $\Psi$ is a 6-sweepout and topological arguments of Lusternik-Schnirelmann, we show that the method above, with some modifications, produces a  free boundary minimal surface with the desired properties that is not the equatorial disk or the critical catenoid.

\subsection{Organization}
We introduce some   preliminaries in \S \ref{minmax}, and in \S \ref{mainproof} prove the  results in \S \ref{intro}. The proofs of some propositions used in \S \ref{mainproof} will be postponed to \S \ref{lemmaproof}.

\subsection*{Acknowledgment}
The author is very grateful to his advisor Andr\'e Neves for the constant support and patient guidance throughout the progress of this work. He would also like to thank  DeVon Ingram, Daniel Mitsutani, Chi Cheuk Tsang,  and Ao Sun for the  helpful conversations, and fedja regarding the proof of Proposition \ref{globalmax}. He  also thanks Danny Calegari for his comments.

\section{Preliminaries}\label{minmax}

In this section, we first discuss more about the space $\mathcal E$ introduced in \S 1, and then state some preliminaries about min-max theory, in both the  Almgren-Pitts setting \cite{Alm62,Alm65,Pit81} and  the Simon-Smith setting \cite{CD03,Smi82}.

\subsection{About the space $\mathcal E$.} \label{trivial}
A smooth embedded surface $S$ in $\mathbb B^3$ is said to be {\it properly embedded} if $\partial S=S\cap\partial\mathbb  B^3$ and $S$ meets $\partial \mathbb B^3$ transversely along $\partial S$. By definition, $\mathcal E$ consists of closed sets $S\subset \mathbb B^3$ such that there exists a finite set $P$ such that $S\backslash P$ is a smooth and properly embedded surface. (Note that  $\partial(S\backslash P)$  does not include $P$.)
Now, for any open set $U\subset\subset \mathbb B^3\backslash P$ (meaning  $\overline{U}\subset \mathbb B^3\backslash P$), $\epsilon>0$, and non-negative integer $k$, denote by $B_{P,U,\epsilon,k}(S)\subset \mathcal E$ the subset of all surface $S'\in\mathcal E$  such that $S'\backslash P$ is smooth and properly embedded and is $\epsilon$-close to $S$  in the graphical $C^k$-distance  within $U$. Then the following proposition tells us that the topology on $\mathcal E$ introduced in \S \ref{intro} is well-defined.
\begin{prop}
The subsets $B_{P,U,\epsilon,k}(S)\subset \mathcal E$ form a base.
\end{prop}
\begin{proof}
First, these subsets clearly cover $\mathcal E$. So it suffices to show that if  $ B_{P_1,U_1,\epsilon_1,k_1}(S_1)\cap B_{P_2,U_2,\epsilon_2,k_2}(S_2)$ contains some element $S$, then it contains some subset $ B_{P,U,\epsilon,k}(S)$. Indeed, one can just take $P:=P_1\cap P_2$, $U:= U_1\cup U_2$, $k:=\max\{k_1,k_2\}$, and $\epsilon>0$ to be sufficiently small. 
\end{proof}

We will mention some mostly obvious remarks.  First, $\mathcal E$ contains disconnected surfaces, and also the empty  surface $\emptyset$ and any finite sets of points tautologically.  Taking $P=\emptyset$ in (\ref{singsub}), we know  $\{\emptyset\}$ is an open subset of $\mathcal E$. However, $\emptyset\in\mathcal E$ is not an isolated point as for any $p\in\mathbb B^3$, all open neighborhoods of $\{p\}$ in $\mathcal E$ has $\emptyset$ as an element tautologically. Similarly, for any distinct points $p_1,p_2\in\mathbb B^3$, all open neighborhoods of $\{p_1,p_2\}$ in $\mathcal E$ has $\{p_1\}$ as an element, but not vice versa. Moreover, for any $p\in\mathbb B^3$, let $B_r(p)\subset\mathbb B^3$ be the ball centered at $p$ with radius $r$. Then for $n\geq 2$, $\partial B_{1/n}(p)\in\mathcal E$ and converge to $\{p\}$  (not $\emptyset$) as $n\to\infty$. Furthermore, the path $r\mapsto \partial B_r(0)$ of spheres in $\mathcal E$ for $r\in (0,2)$ is not well-defined at $r=1$, but by perturbing the spheres to ellipsoids, the path becomes well-defined and continuous.

\subsection{Simon-Smith min-max theory}\label{ssminmax} Let $M$ be a compact oriented Riemannian 3-manifold with strictly mean convex boundary.

\begin{defn}\label{smoothfamily}
Let $X$ be a compact $k$-dimensional cubical complex, called the {\it parameter space}. Suppose we have a map $\Phi$ assigning to each $x\in X$ a closed subset $\Phi(x)$ of $M$ such that:
\begin{enumerate}
    \item There exists a dense subset $Y\subset X$ of parameters such that:
    \begin{itemize}
        \item For each $x\in Y$, $\Phi(x)$  is an oriented, smooth, and properly embedded surface with boundary.
        \item For each $x\in X\backslash Y$, there exists a finite set $P(x)$ such that $\Phi(x)\backslash P(x)$ is a smooth and properly embedded surface with boundary.
    \end{itemize}
    Moreover, we require that $|P(x)|$ is  bounded independent of $x$. (We can say $P(x)=\emptyset$ for $x\in Y$ for convenience.)
    \item  $\Phi$ is continuous in the varifold topology.
    \item \label{localsmooth}
    For any $x_0\in X$ and open set $U\subset\subset M\backslash P(x_0)$ (i.e. $\overline{U}\subset M\backslash P(x_0)$), $\Phi(x)\to \Phi(x_0)$ in the graphical $C^\infty$-topology  in $U$ whenever $x\to x_0$.
    \item $\Phi(x)$ has genus at most $g$ for each $x\in Y$.
\end{enumerate}
Then we call $\Phi$ a {\it smooth family of surfaces with genus at most $g$}, or in brief, a {\it genus $\leq g$ smooth family}.
\end{defn} 

Note that when $\Phi(x)$ is disconnected, its genus is defined as the sum of the genus of each of its connected components.  For (3), $\Phi(x_0)$ meets $\partial \mathbb B^3$ transversely in $U$, thus the graphical convergence makes sense even near the boundary $\partial\Phi(x_0)$. Moreover, we required continuity in the varifold topology (see \cite{CFS20,Fra21}) instead of the Hausdorff topology because we want to allow a smooth family to contain empty sets: We will explain more about the minor variations between our definition of a smooth family and others' later in the proof of Theorem \ref{ssminmaxthm}.

Two smooth families $\Phi$ and $\Phi'$ parametrized by $X$ are said to be {\it homotopic} if 
 there exists a map $\psi\in C^\infty(X\times M,M)$ such that
$\psi(x,\cdot)\in \textrm{Diff}_0(M)$ for each $x$ (meaning each $\psi(x,\cdot)$ is homotopic via diffeomorphisms to the identity map), and $\psi(x,\Phi(x))=\Phi'(x)$ for each $x$.
Given a homotopy class $\Lambda$, its {\it width} is defined by
$$\LSS(\Lambda):=\inf_{\Phi\in \Lambda}\max_{x\in X}\area(\Phi(x)).$$ 
A sequence $\{\Phi_i\}$  in $\Lambda$ is said to be {\it minimizing} if $$\lim_{i\to\infty} \max_{x\in X} \area(\Phi_i(x))=\LSS(\Lambda).$$ 
If $\{\Phi_i\}$ is a minimizing sequence and we pick $x_i$ such that $$\lim_{i\to\infty} \area(\Phi_i(x_i))=\LSS(\Lambda),$$
then $\{\Phi_i(x_i)\}$ is called a {\it min-max sequence}. Furthermore, a minimizing sequence is {\it pulled-tight} if all its min-max sequences approach the set of stationary varifolds  in the varifold topology.
\begin{thm}\label{ssminmaxthm}
Let $\Lambda$ be a homotopy class of genus $\leq g$ smooth families parametrized by $X$. Then     there exists a pulled-tight minimizing sequence in $\Lambda$, which contains a min-max sequence  converging in  the varifold topology  to some varifold $V=\sum^N_{i=1}n_i\Gamma^i$, in which  $\Gamma^i$ are disjoint embedded free boundary minimal surfaces and $n_i$ are positive integers, such that:
\begin{itemize}
    \item $\norm{V}=\LSS(\Lambda)$. 
    \item $\ind(\spt(V))\leq\dim(X)$.
    \item $\displaystyle \sum_{\Gamma^i\;\;\mathrm{ orientable}} \genus(\Gamma^i)+\frac 12 \sum_{\Gamma^i\;\;\mathrm{ non\textrm{-}orientable}} (\genus(\Gamma^i)-1)\leq g$.
\end{itemize} 
\end{thm}
\begin{proof}
It suffices to prove the following statements:
\begin{enumerate}
    \item There exists in $\Lambda$ a  pulled-tight minimizing sequence $\{\Phi_n\}$.
    \item There exists  a  function $r:M\to \R_{>0}$ and a min-max  sequence $\{\Phi_n(x_n)\}$ of the minimizing sequence above such that:
    \begin{itemize}
        \item For every $p\in M$, in every annulus  centered at $p$  with outer radius at most $r(p)$, $\Phi_n(x_n)$ is {\it $1/n$–almost minimizing} (see  \cite[Definition 3.2]{CD03}) when $n$ is large enough.
\item  In any such annulus, $\Phi_n(x_n)$  is smooth  when $n$ is large enough.
\item $\Phi_n(x_n)$ converges to a stationary varifold $V$.
    \end{itemize}
    \item $V$ has the  desired form $\sum^N_{i=1}n_i\Gamma^i$ mentioned above.
    \item The index bound.
    \item The genus bound.
\end{enumerate}

Item (1) follows from the pull-tight procedure in  \cite[\S 4]{CD03} of Colding-De Lellis. Item (2) follows from \cite[Proposition 5.1]{CD03} and its multi-parameter version \cite[Appendix]{CGK18}   by Colding-Gabai-Ketover. For the adaptation to the case of manifold with boundary, see \cite{Li15} by Li and \cite{Fra21} by Franz. Note the following differences between our setting and previous ones. First, our  parameter space $X$ is a cubical complex instead of a cube, but we can embed it into some   cube of high dimension so that the same proofs  work.  Second, even though unlike in \cite{CD03} we are   doing {\it non-relative} min-max theory, as we allow a homotopy to vary a smooth family on the {\it boundary} of its parameter space, the same  argument of \cite{CD03} is still applicable (in the Almgren-Pitts setting, the non-relative version was carried out in \cite{MN17}). Third, in our definition of a smooth family $\Phi$,  we allow the set $P(x)$ of singularities  of $\Phi(x)$ to vary as $x$ varies. However, we can still ensure  each $\Phi_n(x_n)$  to be  smooth in any small annulus  described  in (2). This is because, by passing to a subsequence, we can assume that $P(x_n)$ converges as $n\to\infty$ to some finite set $P$ in the Hausdorff topology, so that  our claim follows immediately  by choosing $r(p)$ to be small enough (see the last paragraph of \cite[\S 5]{CD03}).

As for item (3), the regularity of $V$ is due to  \cite[Theorem 7.1]{CD03} for the closed case and  \cite[Proposition 4.11]{Li15} for the free boundary case: Notice that we have assumed $\partial M$ to be strictly mean convex, which via the maximum principle guarantees that the interior of $V$ does not touch $\partial M$.
As for the index bound, it was first proven in the Almgren-Pitts setting, by Marques-Neves  \cite{MN16} in the closed case and Guang-Li-Wang-Zhou  \cite{GLWZ21} in the free boundary case, and then  adapted to the Simon-Smith setting by Franz \cite{Fra21}.
Finally, the genus bound is due to  \cite[Theorem 9.1]{Li15} by Li, based on \cite[Theorem 1.6]{DP10} by   De Lellis-Pellandini. We note that although the set $X\backslash Y$ of parameters that give non-smooth surfaces may not be finite, the proof of the genus bound (in \cite[\S 2.3]{DP10}) using Simon's lifting lemma \cite[Proposition 2.1]{DP10} is still valid using the fact that the complement $(X\backslash Y)^c=Y$ is dense by assumption. (See also \cite{Ket16b,Ket19} by Ketover, which provide a stronger genus bound for limits of min-max sequences of smooth surfaces.)
\end{proof}

\subsection{Almgren-Pitts min-max theory }\label{apminmax}  Let $M$ be a compact  $(n+1)$-dimensional Riemannian manifold with boundary. 
Let $\mathcal R_k(M;\Z_2)$ (resp. $\mathcal R_k(\partial M;\Z_2)$) be the set of $k$-dimensional rectifiable currents in $M$ (resp. $\partial M$)  with $\Z_2$-coefficients. For any $T\in \mathcal R_k(M;\Z_2)$ such that its support lies in $\partial M$, we define an equivalence relation by $T\sim S$ if $T-S\in \mathcal R_k(\partial M;\Z_2)$, and then denote by 
$\Zc_{k}( M,\partial M;\Z_2)$ the set of such equivalence classes. 
The three common topologies on $\Zc_{k}( M,\partial M;\Z_2)$ are given by the  {\it flat} metric $\Fc$,  the $\Fb$-metric, and the  {\it mass} $\Mb$ respectively: Since the definitions are standard, we refer the reader to, for example, \cite[\S 3]{GLWZ21} for the details. Note that by \cite[\S 3.3]{GLWZ21}, under the metric $\Fc$ or $\Mb$, $\Zc_{k}( M,\partial M;\Z_2)$ is  homeomorphic and isometric to the space of {\it relative $k$-cycles} considered in \cite[\S 2.2]{LMN18}.

Then by the Almgren isomorphism theorem \cite{Alm62} (see also \cite[\S 2.5]{LMN18}), if $H_n( M,\partial M;\Z_2)=0$, then when equipped with the flat topology, $\Zc_{n }( M,\partial M;\Z_2)$ is connected and weakly homotopic equivalent to $\R\mathbb P^\infty$. Thus we can denote  its cohomology ring by $\Z_2[\bar\lambda]$. Then an $\Fc$-continuous map $\Phi:X\to \Zc_{n }( M,\partial M;\Z_2)$, where $X$ is some cubical complex, is said to be a {\it $k$-sweepout} if $\Phi^*(\bar\lambda^k)\ne 0$. Let $\Pc_k$ be the set of all $\Fb$-continuous  $k$-sweepouts. 
Then, denoting by $\dmn(\Phi)$ the domain of $\Phi$, the {\it $k$-width} of $M$ is defined by $$\omega_k(M):=\inf_{\Phi\in\Pc_k}\max_{x\in \dmn(\Phi)}\Mb(\Phi(x)).$$
\begin{rmk}\label{equiv}
There is an equivalent characterization of $k$-sweepouts (see \cite[Definition 4.1]{MN17}): An $\Fc$-continuous map $\Phi:X\to\Zc_{n }( M,\partial M;\Z_2)$ is a $k$-sweepout if there exists an $\lambda\in H^1(X;\Z_2)$ such that:
\begin{itemize}
    \item $\lambda${ \it detects the 1-sweepouts}, i.e. for any cycle $\gamma:S^1\to X$, we have $\lambda(\gamma)\ne 0$ if and only if $\Phi\circ \gamma:S^1\to\Zc_{n }( M,\partial M;\Z_2)$ is a 1-sweepout.
    \item The cup product $\lambda^k\in H^k(X;\Z_2)$ is non-zero.
\end{itemize}
\end{rmk}

\section{Proofs of Main results}\label{mainproof}
In this section, we  prove the results stated in \S \ref{intro}. From now on, we denote by  $\mathcal Z$ the space $\Zc_{2 }(\mathbb B^3,\partial\mathbb B^3;\Z_2)$ with the flat topology.

\subsection{Proof of Theorem \ref{6sweep}}
We will construct the desired family $\Psi$ that satisfies  condition \ref{cond5sweep}, \ref{condgenusbdd}, and \ref{condareabdd}  of Theorem \ref{6sweep} in two steps: In step 1 we construct a 6-sweepout (condition \ref{cond5sweep}). Then in step 2, we modify it such that it  becomes, in addition, a genus $\leq 1$ smooth family (condition \ref{condgenusbdd}) with maximal area less than $2\pi$ (condition \ref{condareabdd}).

\medskip{\noindent \bf Step 1.}
We consider all scalings and translations of the saddle surface $\{x^2-y^2+z=0\}$ in $\R^3$, and then intersect them with $\mathbb B^3$.  Namely, we define a map $\Phi_{4}:\R\mathbb P^4\to\Zc$ by assign to each $a=[a_0:a_1:a_2:a_3:a_4]\in\RP^4$ the zero set of the polynomial
$$p_{a_0,a_1,a_2,a_3,a_4}(x,y,z):=a_0(x^2-y^2)+a_1x+a_2y+a_3z+a_4$$
in $\mathbb B^3$. And then we add in rotations. Namely, we define $\tilde \Phi_7:\mathbb{RP}^4\x SO(3)\to\Zc$ by assigning each $(a,Q)$ to the surface ``$\Phi_4(a)$ rotated by $Q^{-1}$", i.e.  the zero set of the polynomial $p_{a_0,a_1,a_2,a_3,a_4}(Q(x,y,z))$ in $\mathbb B^3$.

However, a loop in the $SO(3)$ factor does not produce a 1-sweepout (e.g. consider a disk rotating for $360^\circ$), and 
$\tilde\Phi_7$ is not yet a 6-sweepout. To get a 6-sweepout, one needs to take a quotient on the space $\RP^4\x SO(3)$ as follows.

We first observe $\{x^2-y^2+z=0\}$ has a dihedral symmetry: Let 
$$g_1:=\left(\begin{smallmatrix}
0 & 1 & 0\\ 
1 & 0 & 0\\ 
0 & 0 & -1
\end{smallmatrix}\right)\;\mathrm{  and  }\;\; g_2:=\left(\begin{smallmatrix}
0 & -1 & 0\\ 
-1 & 0 & 0\\ 
0 & 0 & -1
\end{smallmatrix}\right).$$
They are the $180^\circ$-rotation about the line  $\{z=0, x=y\}$ and $\{z=0, x=-y\}$ respectively. Then $D_2:=\{\id, g_1, g_2, g_1g_2\}$ is a dihedral group, which preserves $\{x^2-y^2+z=0\}$. Motivated by this, we define a $D_2$-action on $\RP^4$ by 
\begin{align*}
&g_1[a_0:a_1:a_2:a_3:a_4]:=[-a_0:a_2:a_1:-a_3:a_4],\\
&g_2[a_0:a_1:a_2:a_3:a_4]:=[-a_0:-a_2:-a_1:-a_3:a_4],
\end{align*}
and then a $D_2$-action on $\RP^4\x SO(3)$ by
\begin{align}\label{adef}
&g_1(a,Q):=(g_1a,g_1Q),\\
&g_2(a,Q):=(g_2a,g_2Q).\nonumber
\end{align}
The whole reason we define the action by (\ref{adef}) is to ensure the following:
\begin{prop}\label{phiprod}
For each $g\in D_2$ and $a\in\RP^4$,  $g^{-1}(\Phi_4(g(a)))=\Phi_4(a)$.
\end{prop}
\begin{proof}
The proof is straightforward: Let $a=[a_0:a_1:a_2:a_3:a_4]$. Then
\begin{align*}
    g_1^{-1}(\Phi_4(g_1(a)))&=\{-a_0(y^2-x^2)+a_2y+a_1x-a_3(-z)+a_4=0\}\cap\mathbb B^3,\\
    g_2^{-1}(\Phi_4(g_2(a)))&=\{-a_0((-y)^2-(-x)^2)-a_2(-y)-a_1(-x)-a_3(-z)+a_4=0\}\cap\mathbb B^3,
\end{align*}
which are both the same as $\Phi_4(a)$.
\end{proof}
As an immediate result,  $\tilde\Phi_7(g(a,Q))=\tilde\Phi_7(a,Q)$ for all $g,a,$ and $Q$, and hence one can pass to the quotient space to define a new collection $\Phi_7:\frac{ \mathbb{RP}^4\x SO(3)}{D_2}\to\Zc$. Note that $\Phi_7$ is $\Fc$-continuous clearly because it parametrizes the collection of all scalings, translations, and rotations of the saddle surface $\{x^2-y^2+z=0\}$, intersected with $\mathbb B^3$.
Then a crucial fact is:

\begin{prop}\label{6sweepout} $\Phi_7$ is a 6-sweepout.
\end{prop}
The proof of Proposition \ref{6sweepout} is a lengthy calculation of algebraic topology. We postpone it to \S \ref{6sweepoutproof}. Hence, $\Phi_7$ satisfies condition \ref{cond5sweep}.

\medskip{\noindent \bf Step 2.} However, $\Phi_7$ is  not a  smooth family, as it contains intersecting disks of the form $\{(x-x_0)^2-(y-y_0)^2=0\}\cap\mathbb B^3$. We are going to desingularize them.

For each fixed number $a_5\geq 0$, let us define a collection $\Phi^{a_5}_4: \RP^4\to \Zc$ by assigning $[a_0:a_1:a_2:a_3:a_4]$ to the zero set of the polynomial 
\begin{equation}\label{poly}
    p_{a_0,a_1,a_2,a_3,a_4,a_5}(x,y,z):= a_0(x^2-y^2+{a_5} z^3)+a_1x+a_2y+a_3z+a_4
\end{equation}
in $\mathbb B^3$. So for small $a_5>0$, $\Phi^{a_5}_4$ is slight modification of $\Phi_4$. In fact, as we will see, {\it the effect of the $z^3$ term is three-fold:  Desingularizing the intersecting disks, creating genus $1$ surfaces (Figure \ref{fig:phi1}), and lowering the area to strictly  below $2\pi$}. 

\begin{figure}
    \centering
    \makebox[\textwidth][c]{\includegraphics[width=6in]{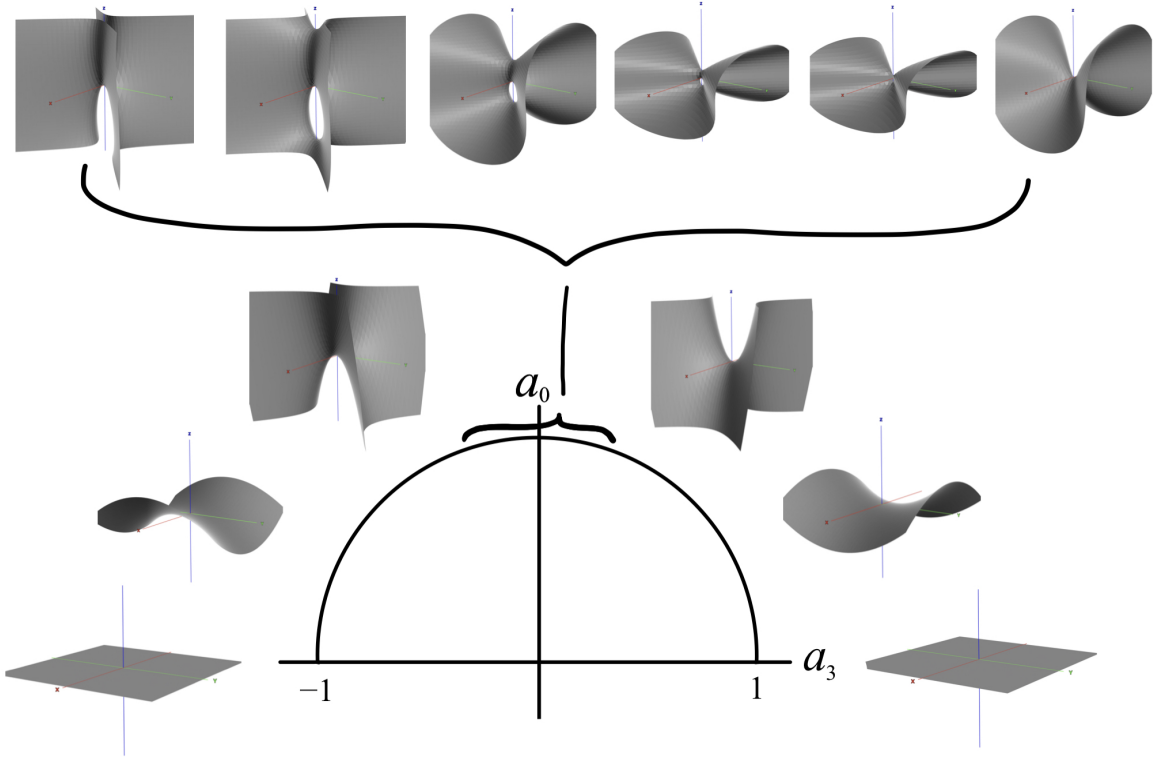}}
    \caption{The surface $\{a_0(x^2-y^2+a_5z^3)+a_3z=0\}$ for various $a_0$ and $a_3$, with $a_5>0$ small and fixed.}
    \label{fig:phi1}
\end{figure}

\begin{rmk}\label{beh}
Let us geometrically describe the family $\Phi_4^{a_5}$. We will focus on the cubic surfaces, so without loss of generality we put $a_0=1$. Since $a_1$ and $a_2$ just contribute to translation, let us assume they are both 0. Now, fix some $(a_3,a_4,a_5)\in \R^2\x(0,\infty)$, and let $s>0$ varies. We claim that as $s$ increases from 0, the surfaces
\begin{equation}\label{a5s}
    \Phi_4^{a_5 s}([1:0:0:a_3s:a_4s])=
\{x^2-y^2+s(a_3z+a_4+a_5z^3)=0\}\cap \mathbb B^3
\end{equation} {\it desingularize the intersecting disks $\{x^2-y^2=0\}\cap \mathbb B^3$ along the singular line.}
Indeed, we consider the three cases: $a_3z+a_4+a_5z^3$ having 1, 2, or 3 roots; let $z_i$'s be the roots. Then in each case, as $s$ increases from 0, the surfaces $\Phi_4^{a_5 s}([1:0:0:a_3s:a_4s])$ stay fixed on the coordinate planes $\{z=z_i\}$, and ``open up" smoothly above, below, and in between. This is because by (\ref{a5s}), at each fixed height $z\ne z_i$, the cross-section of the surface is a hyperbola, which  dilates as $s$ increases and has a distance of $\sqrt{s|a_3z+a_4+a_5z^3|}$ between the two branches. In Table \ref{table::openup}, we show some examples. 

Moreover, we can study for which $(a_3,a_4,a_5)$ the surface $\Phi_4^{a_5}([1:0:0:a_3:a_4])$ has singularities, and where they are: By solving
$$
\begin{cases}
  p_{1,0,0,a_3,a_4,a_5}=x^2-y^2+a_3z+a_4+a_5z^3=0 \text{  } \\
  \nabla p_{1,0,0,a_3,a_4,a_5}=(2x,-2y,a_3+3a_5z^2)=(0,0,0) \text{  }
\end{cases}
$$
we know when $a_3z+a_4+a_5z^3$ has some double or triple root $z_1$, the surface has a singularity at $(0,0,z_1)$.

\begin{table}
\begin{tabular}{ |c|c|c|c| } 
  \hline
  \backslashbox{$s$}{$b_3$} & 0.1 & $-3\sqrt[3]{\frac 1{400}}\approx 0.407$ &0.6 \\ 
  \hline
  0.05 & 
\includegraphics[width=1.5in]{11.pdf}
& 
\includegraphics[width=1.5in]{12.pdf}& 
\includegraphics[width=1.5in]{13.pdf}\\ 
  \hline
  0.3 & \includegraphics[width=1.5in]{21.pdf} & \includegraphics[width=1.5in]{22.pdf} & \includegraphics[width=1.5in]{23.pdf} \\ 
  \hline
  \end{tabular}
  \caption{The surface $\{x^2-y^2+s(b_3z+0.1+z^3)=0\}$ for various $b_3$ and $s$ is shown above. They illustrate the three cases where the polynomial $a_3z+a_4+a_5z^3$ has 1, 2, and 3 roots respectively.}
  \label{table::openup}
\end{table}

\end{rmk}

But our goal is to modify $\Phi_7$, not just $\Phi_4$. To do that, first notice by a straightforward calculation that $\Phi^{a_5}_4$ satisfies a property similar to Proposition \ref{phiprod}, namely $g^{-1}(\Phi^{a_5}_4(g(a)))=\Phi^{a_5}_4(a)$. The key idea behind is that {\it the polynomial   $x^2-y^2+a_5z^3$ is invariant under the $D_2$-action on $(x,y,z)$.} As a result, we can construct a map $\Phi^{a_5}_7:\frac{ \mathbb{RP}^4\x SO(3)}{D_2}\to\Zc$ by rotating all elements in $\Phi^{a_5}_4$, just like how we constructed $\Phi_7$ from $\Phi_4$ in step 1. Hence we have obtained a modification $\Phi^{a_5}_7$ of $\Phi_7$.

Now, from Remark \ref{beh}, it follows easily that $\Phi^{a_5}_7$ is $\Fb$-continuous and is homotopic in the $\Fc$-topology to $\Phi_7$, so that it is a 6-sweepout (condition \ref{cond5sweep}). Moreover, we show that condition \ref{condgenusbdd} can be satisfied:
\begin{prop}\label{genus01} For almost every  $a_5\in [0,1]$, 
$\Phi^{a_5}_7$ is a genus$\;\leq 1$ smooth family.
\end{prop} 
\begin{proof}
By Remark \ref{beh}, the subset of parameters $a\in \frac{\RP^4\x SO(3)}{D_2}$ such that $\Phi^{a_5}_7(a)$ has singularities is 1-codimensional, and all such surfaces have at most one singularity. 
Moreover, it is straightforward to check that for all $a$,  $\Phi^{a_5}_7(a)$ and $\partial \mathbb B^3$ touch (i.e. have coinciding tangent planes) at finitely many points. Also, by the  transversality theorem and Remark \ref{beh}, for a.e. $a_5\in[0,1]$,  the algebraic surfaces in $\mathbb R^3$ that define $\Phi^{a_5}_7(a)$ (i.e. the rotations of the zero set of (\ref{poly}) in $\mathbb R^3$) intersect $\partial \mathbb B^3$ transversely for a.e. $a\in \frac{\RP^4\x SO(3)}{D_2}$. So for such $a_5$,  $\Phi^{a_5}_7$ satisfies Definition \ref{smoothfamily} (1).

Note that each smooth surface of $\Phi^{a_5}_4$ has genus 0 or 1 because they are obtained from opening up the intersecting disks $\{x^2-y^2=0\}\cap\mathbb B^3$ above, below, and in between {\it at most three horizontal planes}, by Remark \ref{beh}. So each smooth surface of $\Phi^{a_5}_7$ has genus 0 or 1. Now, using  Remark \ref{beh}, one can show that  Definition \ref{smoothfamily} (2) and (3) are also satisfied by $\Phi^{a_5}_7$. So $\Phi^{a_5}_7$ is a genus $\leq 1$ smooth family for a.e. $a_5$. \end{proof}

Next, we claim that for small $a_5>0$, $\Phi^{a_5}_7$ also satisfies condition \ref{condareabdd}, i.e. 
 $\area\circ\Phi^{a_5}_7<2\pi$. Indeed, it suffices to show that $\area\circ\Phi^{a_5}_4<2\pi$ for small $a_5>0$, which follows straightforwardly from the following two propositions: 
\begin{prop}\label{globalmax} The area function $\area\circ\Phi_4:\RP^4\to\R$ attains a strict global maximum at $[1:0:0:0:0]$.
\end{prop}
Note that $\Phi_4([1:0:0:0:0])$ is $\{x^2-y^2=0\}\cap\mathbb B^3$, which has area $2\pi$.
\begin{prop}\label{localmax}
Define $\Phi_5:\RP^4\x[0,1]\to\Zc$ by $\Phi_5(a,a_5)=\Phi^{a_5}_4(a)$. Then the area function $\area\circ\Phi_5$ attains a strict local maximum at $([1:0:0:0:0],0)$.
\end{prop}

The proof of Proposition \ref{globalmax} is due to the MathOverflow user fedja \cite{fed}. We include it in Appendix \ref{globalmaxproof}. The proof of Proposition \ref{localmax} is postponed to \S \ref{localmaxproof}: It uses mainly calculus but is quite technical. However, intuitively Proposition \ref{localmax} makes sense because Remark \ref{beh} tells us that $\Phi_5$ gives  {\it a desingularization of the intersecting equatorial disks, and desingularization should lower the area as the sharp bend along the singular line is smoothed.}

Thus, for a.e. sufficiently small $a_5>0$, by Proposition \ref{globalmax} and \ref{localmax},  $\Phi^{a_5}_7$  satisfies condition \ref{condareabdd} also. Defining $\Psi$ as one such $\Phi^{a_5}_7$, we  finish the proof of Theorem \ref{6sweep}.

\subsection{Proof of Theorem \ref{type}} \label{typeproof} We will first do the case of $\overline{\mathcal E_0\cup \mathcal E_1}\cap \mathcal E^{2\pi}$. By  Theorem \ref{6sweep} \ref{condgenusbdd}  and \ref{condareabdd}, we can view the family $\Psi$ in Theorem \ref{6sweep} as the composition of a map $\Psi':\frac{ \mathbb{RP}^4\x SO(3)}{D_2}\to \overline{\mathcal E_0\cup \mathcal E_1}\cap \mathcal E^{2\pi}$ and the  natural map $i$ from   $\overline{\mathcal E_0\cup \mathcal E_1}\cap \mathcal E^{2\pi}$ to the space of 2-cycles $\Zc_{2}(\mathbb B^3,\partial \mathbb B^3;\Z_2)$ equipped  with the flat topology (see \S \ref{apminmax}). Note that $\Psi'$ is continuous (under the topology of $\mathcal E$ defined in \S \ref{intro}) since $\Psi$ is a smooth family by Theorem \ref{6sweep} \ref{condgenusbdd}, and $i$ is continuous by the fact that the smooth convergence of surfaces is stronger than the flat convergence.

By Almgren isomorphism theorem (see \S \ref{apminmax}), we  denote the cohomology ring of $\Zc_{2}(\mathbb B^3,\partial \mathbb B^3;\Z_2)$ in $\Z_2$-coefficients as $\Z_2[\bar\lambda]$. To prove Theorem \ref{type} for the space $\overline{\mathcal E_0\cup \mathcal E_1}\cap \mathcal E^{2\pi}$, it suffices to show that $(i^*\bar\lambda)^6\ne 0$. Thus, it suffices to show $(i\circ \Psi')^*(\bar\lambda^6)\ne 0$, i.e. $\Psi$ is a 6-sweepout, which is true by Theorem \ref{6sweep}.

Reusing the argument above with $\overline{\mathcal E_0\cup \mathcal E_1}\cap\mathcal E^{2\pi}$ replaced by any subspace of $\mathcal E^{\infty}$ that contains $\overline{\mathcal E_0\cup \mathcal E_1}\cap\mathcal E^{2\pi}$,  we finish the proof of  Theorem \ref{type}. (Note that elements in $\mathcal E^{\infty}$ have finite area and thus belong to $\Zc_{2}(\mathbb B^3,\partial \mathbb B^3;\Z_2)$.)

\subsection{Proof of Corollary \ref{6width}} Let $\Psi$ be the smooth family in Theorem \ref{6sweep}. Since $\area\circ\Psi<2\pi$ and $\Psi$ is a 6-sweepout by Theorem \ref{6sweep},   $\omega_6(\mathbb B^3)<2\pi.$

\subsection{Proof of Theorem \ref{min1}}
Let $\Psi$ be the family satisfying condition \ref{cond5sweep}, \ref{condgenusbdd}, and \ref{condareabdd} of Theorem \ref{6sweep}, and $\Psi^{(5)}$ be the subfamily of $\Psi$ parametrized by a 5-skeleton of the parameter space of $\Psi$. Without loss of generality, we can assume that  $\Psi^{(5)}$ is also a smooth family. Now, since $\Psi$ is a 5-sweepout by \ref{cond5sweep}, so is $\Psi^{(5)}$ (see  the proof of \cite[Proposition 7.1]{MN21}). It follows that the width   $L:=\LSS(\Lambda(\Psi^{(5)}))$ is positive.

Now, we  apply the Simon-Smith min-max theory to $\Psi^{(5)}$. Let $\{\Phi_i\}$ be a pulled-tight minimizing sequence of $\Lambda(\Psi^{(5)})$. Denote by $\mathcal W$ the set of all stationary integral varifolds in $\mathbb B^3$ whose  support is a smooth  embedded free boundary minimal hypersurface, and by   $\mathbf C(\{\Phi_i\})$ the set of subsequential varifold limits  of min-max sequences of $\{\Phi_i\}$. 
Then by Theorem \ref{ssminmaxthm}, $\mathbf C(\{\Phi_i\})\cap\mathcal W$ is non-empty. 
Now, there are three cases: $\mathbf C(\{\Phi_i\})\cap\mathcal W$ contains (1) some element $\Gamma$ that is not the equatorial disk or the critical catenoid; (2) only critical catenoids; or (3) only equatorial disks (note that critical catenoids and equatorial disks cannot appear together in $\mathbf C(\{\Phi_i\})\cap\mathcal W$ as they have different area). We will consider each case individually in the following.

\medskip{\noindent \bf Case (1).} We will show that $\Gamma$ has the desired property stated in Theorem \ref{min1}:
\begin{prop}\label{desired}
$\Gamma$ has multiplicity 1, genus $0 $ or $1$, Morse index $4$ or $5$, and area in the range $(\pi,2\pi)$.
\end{prop}
\begin{proof}
First, by Theorem \ref{ssminmaxthm}, $\area(\Gamma)$ when counted with  possible multiplicities is equal to $\LSS(\Lambda(\Psi^{(5)}))$, which is less than $2\pi$ by \ref{condareabdd}. Then, since  the least possible  area of a free boundary minimal surface in $\mathbb B^3$ is $\pi$ by a result of Fraser-Schoen \cite[Theorem 5.4]{FS11},  $\Gamma$ must have multiplicity 1.
Moreover, by Theorem \ref{ssminmaxthm}, \ref{condgenusbdd} implies $\genus(\Gamma)\leq 1$, and $\ind(\Gamma)\leq 5$ since the parameter space of $\Psi^{(5)}$ is 5-dimensional.
Lastly, since $\Gamma$ is not the equatorial disk, we have $\area(\Gamma)>\pi$ again from  \cite{FS11} (and also \cite{Bre12} by Brendle), and $\ind(\Gamma)\geq 4$  from \cite[\S 5]{Dev19} by Devyver or \cite[\S 3.1]{Tra20} by Tran.
\end{proof}

Hence, case (1) is done. 

\medskip{\noindent \bf Case (2).} Now we turn to case (2). We will use a technique called {\it splitting of domains}. First, let $\mathcal C$ denote  the set of critical catenoids: Note that $\mathcal C$ is homeomorphic to $\RP^2$. Fixing a small $\epsilon>0$, and denoting by $W$ the parameter space of $\Psi^{(5)}$, we consider the set of parameters
$$\{x\in W:\Fb(\Phi_i(x),\mathcal C)\leq \epsilon\}.$$
This subset, after a slight thickening, can be assumed to be a cubical complex; we will denote it by $Z_i$, and then  $\overline{W\backslash Z_i}$ by $Y_i$. Now, as we mentioned $\Psi^{(5)}$ is a 5-sweepout, hence so is $\Phi_i$. Then using a topological argument by Lusternik-Schnirelmann \cite{LS47} (see also \cite[Claim 6.3]{MN17}), we know that either $\Phi_i|_{Z_i}$ is a 1-sweepout or $\Phi_i|_{Y_i}$ is a 4-sweepout. Now note that, if $\epsilon$ is small enough, $\Phi_i|_{Z_i}$ lies near $\mathcal C$ and thus is homotopic in the $\Fc$-topology to some $\Mb$-continuous map into $\mathcal C$ (using \cite[\S 3.3.6]{Nur16} by Nurser, together with discretization and interpolation theorems in the free boundary setting  by Li-Zhou \cite[\S 4.2]{LZ21}). But no map into $\mathcal C$ can be a 1-sweepout as $\mathcal C$ can be contracted to just $\{\emptyset\}$, by 
shrinking each critical catenoid to its axis, which has no mass. Hence, $\Phi_i|_{Z_i}$ cannot be a 1-sweepout, and so {\it each $\Phi_i|_{Y_i}$ must be a 4-sweepout.}

Now, we  claim that { \it for some $i$,  $\Phi_i|_{Y_i}$ is homotopic (in the Simon-Smith setting) to another smooth family $\tilde \Psi$ with maximal area less than $L$.} Indeed, if not, then by standard Simon-Smith min-max theory, there exists $y_i$ such that $\Phi_i|_{Y_i}(y_i)$ converges subsequentially to some smooth embedded free boundary minimal surface $V$ with mass $L$ (see \cite[\S 5]{CD03} and \cite[Appendix]{CGK18}: Their arguments apply here because even though our parameter spaces $Y_i$ depend on $i$, they can all be embedded into some  $\R^N$ with $N$ independent of $i$). Then note two facts: $V$ cannot be the critical catenoid by the definition of $Y_i$, and $V\in \mathbf C(\{\Phi_i\})\cap\mathcal W$ clearly. However, these two facts are contradictory because we are in case (2). Thus, the desired smooth family $\tilde\Psi$ exists. 

We then apply the Theorem \ref{ssminmaxthm} to $\Lambda(\tilde\Psi)$, and repeat the  argument above. Namely, letting $\{\tilde\Phi_i\}$ be a pulled-tight minimizing sequence of $\Lambda(\tilde\Psi)$, there are two cases: $\mathbf C(\{\tilde\Phi_i\})\cap\mathcal W$ either contains (2a) some element $\tilde{\Gamma}$ that is not the equatorial disk or the critical catenoid, or (2b) only equatorial disks. The critical catenoid, having area $L$, cannot appear because   $\area\circ\tilde\Psi<L$.

If it is case (2a), one can reapply the proof of Proposition \ref{desired} to $\tilde{\Gamma}$ to show that $\tilde{\Gamma}$ has the desired properties in Theorem \ref{min1} (this time  $\ind(\tilde{\Gamma})$ is actually 4), and we are done. If it is case (2b), we {\it split the domains} again to arrive at a contradiction. This time the key ideas are: There is no $3$-sweepout near the set of equatorial disks, which is merely an $\RP^2$; and there is no $1$-sweepout with maximal area less than $\pi$, which is the least possible area for a free boundary minimal surface. Therefore case (2b) is impossible. Now case (2) is also done.

\medskip{\noindent \bf Case (3).} Case (3)  is entirely analogous to case (2b). 

\medskip
So we have finished the proof of Theorem \ref{min1}.

\subsection{Explanation of Remark \ref{cfs}}\label{cfsproof}
Letting $a_5>0$ be sufficiently small, we define the  family $\Phi_1:\RP^1\to \Zc$ of $D_2$-symmetric surfaces:
\begin{align*}
    \Phi_1([a_0:a_3])&:=\Phi^{a_5}_4([a_0:0:0:a_3:0])=\{a_0(x^2-y^2+a_5z^3)+a_3z=0\}\cap \mathbb B^3
\end{align*}
(see Figure \ref{fig:phi1}). Note that $\area\circ\Phi_1<2$ as  $\area\circ\Phi^{a_5}_4<2$ by Theorem \ref{6sweep} \ref{condareabdd}. Then applying the equivariant Simon-Smith min-max theorem to $\Phi_1$, we obtain a free boundary minimal surface. To show it has the desired properties mentioned in Remark \ref{cfs}, we just proceed in a way similar to \cite[\S 4]{CFS20}. In particular, we can use the proof of \cite[Lemma 4.1]{CFS20} to show that the number of boundary component of the minimal surface obtained is one. Indeed, we first note that for each surface $\Phi_1(a)$, the complement of the three axes of rotations (the $z$-axis, $\{z=0,x=y\}$, and $\{z=0,x=-y\}$) in $\Phi_1(a)$ are topological disks. And this fact is what one need to carry out the proof of \cite[Lemma 4.1]{CFS20}.

\section{Technical Ingredients}\label{lemmaproof}
In this section, we prove Proposition \ref{6sweepout} and \ref{localmax}.

\subsection{Proof of Proposition \ref{6sweepout}}\label{6sweepoutproof}
Throughout \S \ref{6sweepoutproof}, we 
write $X:=\frac{\RP^4\x SO(3)}{D_2}$. The proof has four steps. In step 1 we compute $H_1(X;\Z_2)$, and understand which first homology classes  give 1-sweepouts under $\Phi_7$. In step 2 we find the cohomology class $\lambda\in H^1(X;\Z_2)$ that {\it detects the 1-sweepouts} (as explained Remark \ref{equiv}), and understand its Poincar\'e dual. In step 3 we
show that $\lambda^6\ne 0$. And  in step 4, we prove a technical lemma used in step 3. By Remark \ref{equiv}, 
we immediately obtain the desired claim that $\Phi_7$ is a 6-sweepout.

\medskip{\noindent \bf Step 1.}  Let us first find $\pi_1(X)$. Let $Q_8$ be the {\it quaternion group} $\{\pm1, \pm i, \pm j, \pm k\}$, contained in the group $S^3$ of unit quaternions.
\begin{lem}\label{pi1}
$\pi_1(X)=\Z_2\x Q_8.$
\end{lem}
\begin{proof}
First,  the universal cover of $\RP^4\x SO(3)$ is $S^4\x S^3$; in fact, without loss of generality one may assume the double covering $S^3\to SO(3)$ maps $\pm i$ to $g_1$ and $\pm j$ to $g_2$, and thus $Q_8$ to $D_2$. Then to prove the lemma,   it  suffices to construct a $\Z_2\x Q_8$-action on $S^4\x S^3$ that descends, under the projections $S^4\x S^3\to\RP^4\x SO(3)$ and $\Z_2\x Q_8\to 1\x D_2$, to the $D_2$-action on $\RP^4\x SO(3)$ defining $X$.

First, we define a $ Q_8$-action on $S^4\x S^3$ by
$$(\pm i)\cdot((a_0,a_1,a_2,a_3,a_4),q):=((-a_0,a_2,a_1,-a_3,a_4),\pm iq),$$
$$(\pm j)\cdot((a_0,a_1,a_2,a_3,a_4),q):=((-a_0,-a_2,-a_1,-a_3,a_4),\pm jq).$$
Then, let $\Z_2$ act on $S^4\x S^3$  by  acting antipodally on {\it only  the $S^4$ factor}. After checking  these two actions commute, we obtain a $\Z_2\x Q_8$-action on $S^4\x S^3$, and it is straightforward to check that this action has the desired property.
\end{proof}

Now,  abelianizing $\pi_1(X)=\Z_2\x Q_8$, we have $H_1(X;\Z)=\Z_2\x D_2$, which then by the universal coefficient theorem gives
$    H_1(X;\Z_2)= \Z_2 \x D_2.$
In fact, some first homology classes can be described explicitly as follows. Denote $e_0:=(1,0,0,0,0)$, and let $\tilde{x}_0:=(e_0,1)$ be the base point in $S^4\x S^3$. Then consider the path 
$$\{((a_0,0,0,\sqrt{1-a_0^2},0),1):-1\leq a_0\leq 1\}$$
in $S^4\x S^3$ joining $\tilde x_0$ to $(-e_0,1)$, a path joining $\tilde x_0$ to  $(e_0,i)$, and a path joining $\tilde x_0$ to $(e_0,j)$. Call the projection of these three paths onto $X$, which are actually loops, $c_1,c_2$ and $c_3$ respectively. Then
$[c_1]=(1,\id)$, $[c_2]=(0,g_1)$, and $[c_3]=(0,g_2)$
in $H_1(X;\Z_2)=\Z_2\x D_2$, and hence they  form a base.

\begin{lem}\label{1sweepout}
$(1,\id), (0,g_1), (0,g_2),(1,g_1g_2)$ are exactly the homology classes that give 1-sweepouts under $\Phi_{7}$.
\end{lem}
\begin{proof} It suffices to show that $(1,\id), (0,g_1), (0,g_2)$ give 1-sweepouts.
To show that $(1,\id)$ gives a 1-sweepout, note that $$\Phi_7\circ c_1=\Phi_4(\{[a_0:0:0:a_3:0]:a_0^2+a_3^2=1\}).$$
But in $\RP^4$ the loop $\{a_0^2+a_3^2=1\}$ is homotopic to $\{a_3^2+a_4^2=1\}$, which under $\Phi_4$ gives the collection of all horizontal planes together with the empty set, and that certainly is a 1-sweepout. So $\Phi_7\circ c_1$ is a 1-sweepout.

To show that $(0,g_1)$ gives a 1-sweepout, note that $\Phi_{7}\circ c_2$ gives the motion of rotating the intersecting planes $\{x^2-y^2=0\}$ about the axis  $\{z=0, x=y\}$ by $180^\circ$. Let us call $\{|y|>|x|\}$ the inside region of $\{x^2-y^2=0\}$, and $\{|y|<|x|\}$ the outside. Then the rotation {\it switches the inside and the outside}.
So  $\Phi_{7}\circ c_2$ is a 1-sweepout.

The proof that $(0,g_2)$ gives a 1-sweepout is similar.
\end{proof}

\medskip{\noindent \bf Step 2.} By the universal coefficient theorem, 
$
    H^1(X;\Z_2)= \Hom(H_1(X;\Z_2),\Z_2).
$
Since $[c_1], [c_2], [c_3]$ form a base of $H_1(X;\Z_2)$, we can define respectively their Hom-duals  $\lambda_i:=[c_i]^* \in H^1(X;\Z_2)$ for $i=1,2,3.$ Let $\lambda=\lambda_1+\lambda_2+\lambda_3$, then by Lemma \ref{1sweepout}, $\lambda$ is the cohomology class that detects exactly the 1-sweepouts.
Hence, to prove Proposition \ref{6sweepout}, we  need $\lambda^6\ne 0$. We are going to prove this by considering the Poincar\'e dual of $\lambda^6$, so let us first understand the Poincar\'e dual $PD(\lambda_i)\in H_6(X;\Z_2)$ of $\lambda_i$.

In the remaining of \S \ref{6sweepoutproof}, {\it we will  view $X$ as an $\RP^4$-bundle over the base $B:=SO(3)/D_2$}, and let $p:X\to B$ be the projection. Let $A_0$ be the 6-dimensional subbundle of $X$ over $B$ on which  $a_0=0$. Note that $A_0$ is well-defined because the subset $\{a_0=0\}$ of $\RP^4\x SO(3)$ is $D_2$-invariant.

\begin{lem}\label{PD1}  $PD(\lambda_1)=[A_0]$ in $H_6(X;\Z_2)$.
\end{lem}
\begin{proof}
This is because the loop $c_1$ intersects $A_0$ at only one point in $X$, $D_2\cdot ([0:0:0:1:0],\id).$
\end{proof}

To construct $PD(\lambda_2+\lambda_3)$, we will need to know the cohomology groups of $B$, which is $S^3/Q_8$. We  quote the result \cite[Theorem 2.2 (1)]{TZ08} of Tomoda-Zvengrowsk:
\begin{prop}\label{cohoring} The cohomology ring $H^*(S^3/Q_8;\Z_2)$ is given by
$$\Z_2[\alpha_1, \alpha'_1, \alpha_2, \alpha'_2, \alpha_3]/\sim,$$
in which the subscript of each generator denotes its degree, and  $\sim$ denotes the following  equivalence:
\begin{align*}
 &\alpha^2_1=\alpha_2+\alpha'_2,\;
 \alpha_1\alpha'_1=\alpha'_2,\;
 (\alpha'_1)^2=\alpha_2, \\
 & \alpha_1 \alpha_2=\alpha_1\alpha'_2= \alpha'_1\alpha'_2=\alpha_3,\; \alpha'_1\alpha_2=0,\\
 & \textrm{products of cohomology classes with total degree greater than 3 is 0.}
\end{align*}
\end{prop}

Now, from the definition of $c_2$ and $c_3$,  we know $p\circ c_2$ and $p\circ c_3$ form a base in $H_1(B;\Z_2)\cong \Z_2\x\Z_2$. Remembering $^*$  denotes the Hom-dual, we let $b_2, b_3$ be 2-dimensional submanifolds in $B$ such that
$$[b_2]=PD([p\circ c_2]^*),[b_3]=PD([p\circ c_3]^*)\in H_2(B;\Z_2).$$
We moreover assume $b_2$ and $b_3$ intersect transversely, and define $b:=b_2\cup b_3$.
Let $X|_{b_2}$ be  the restriction  of the $\RP^4$-bundle $X$ over the base $b_2$, and  similarly for $X|_{b_3}$ and $X|_{b}$. Note that they are 6-dimensional.

\begin{lem}\label{PD2} In $H_6(X;\Z_2)$ we have:
\begin{enumerate}
\item  $PD(\lambda_2)=[X|_{b_2}]$.
\item  $PD(\lambda_3)=[X|_{b_3}]$.
\item  $PD(\lambda_2+\lambda_3)=[X|_b]$.
\end{enumerate}
\end{lem}
\begin{proof} Since $\lambda_2$ is the Hom-dual of $[c_2]$,  to show that $[X|_{b_2}]=PD(\lambda_2)$, it suffices to show that the intersections number of $X|_{b_2}$ with $c_1,c_2$ and $c_3$ respectively are $0,1$, and 0. Indeed, this is true because, respectively, $b_2$ can be perturbed to avoid the point $p\circ c_1$ (in $B$), the intersection number of $b_2$ with $p\circ c_2$ is 1, and the intersection number of $b_2$ with $p\circ c_3$ is 0.

 Similarly, one can show $PD(\lambda_3)=[X|_{b_3}]$, and thus $PD(\lambda_2+\lambda_3)=[X|_{b_2}]+[X|_{b_3}]=[X|_b].$
\end{proof}

\medskip{\noindent \bf Step 3.}
Note that $$\lambda^6
=\lambda_1^6+\lambda_1^4(\lambda_2+\lambda_3)^2+\lambda_1^2(\lambda_2+\lambda_3)^4+(\lambda_2+\lambda_3)^6.$$ To show $\lambda^6\ne 0$, it suffices to show the following lemma.
\begin{lem}\label{intersect}
In the cohomology ring $H^*(X;\Z_2)$ we have:
\begin{enumerate}
    \item $(\lambda_2+\lambda_3)^3=0$.
    \item $\lambda_1^4(\lambda_2+\lambda_3)^2=1$.
    \item $\lambda_1^5=0$.
\end{enumerate}

\end{lem}
\begin{proof}
To prove (1), it suffices to perturb three copies of $X|_{b}$ and show that their intersection is empty. To achieve this, we can just perturb three copies of the base $b\subset B$. But their intersection number will be 0, because by Proposition \ref{cohoring} any element of $H^1(B;\Z_2)$ cubes to 0. Hence we have proven (1).

To prove (2) and (3), we need to find different representatives of $[A_0]$. Let $A_1,A_2,A_3$, and $A_4$ be the subbundle of $X$ over $B$ on which $a_1=a_2$,  $a_1=-a_2$, $a_3=0$, and $a_4=0$ respectively.

\begin{lem}\label{a0a1} In $H_6(X;\Z_2)$ we have:
\begin{enumerate} 
    \item $[A_1]=[A_0]+[X|_{b_3}].$
    \item $[A_2]=[A_0]+[X|_{b_2}]$.
    \item $[A_3]=[A_0]$.
    \item $[A_4]=[A_0]+[X|_{b}]$.
\end{enumerate}
\end{lem}

We postpone the proof of  Lemma \ref{a0a1} to step 4. 

To prove (2) of Lemma \ref{intersect}, first note that $A_0\cap A_1\cap A_2\cap A_4$ is the $[0:0:0:1:0]$-bundle over $B$. Also, $b$ can be perturbed to some $\tilde b$ such that $b\cap \tilde b$ is non-trivial in $H_1(B;\Z_2)$, because Proposition \ref{cohoring} says the square of any element in $H^1(B;\Z_2)$ is non-trivial. As a result,  $A_0\cap A_1\cap A_2\cap A_4\cap X|_b\cap X|_{\tilde{b}}$ is non-trivial in $H_1(X;\Z_2)$. Hence, writing $\mu:=\lambda_2+\lambda_3$ for simplicity, we have
\begin{align}\label{longform}
    1&= PD(A_0)PD( A_1)PD( A_2)PD( A_4)PD (X|_b)^2\\
    &=\lambda_1(\lambda_1+\lambda_3)(\lambda_1+\lambda_2)(\lambda_1+\mu)\mu^2\nonumber\\
    &=\lambda_1^4\mu^2+\lambda_1^3(\lambda_2+\lambda_3+\mu)\mu^2
    +\lambda_1^2(\lambda_2\lambda_3+\lambda_2\mu+\lambda_3\mu)\mu^2+\lambda_1\lambda_2\lambda_3\mu^3\nonumber\\
    &=\lambda_1^4\mu^2.\nonumber
\end{align}
The first equality above is from Lemma \ref{a0a1}. The last equality holds because $\lambda_2+\lambda_3+\mu,\lambda_2\lambda_3+\lambda_2\mu+\lambda_3\mu,$ and $\lambda_2\lambda_3\mu$ all are zero: This is straightforward to check by considering how the bases $b_2,b_3$, and $b$ intersect using Proposition \ref{cohoring}. Hence, we have proven (2) of Lemma \ref{intersect}.

To prove (3) of Lemma \ref{intersect}, one  note that $A_0\cap A_1\cap A_2\cap A_3\cap A_4$ is empty, and then $\lambda_1^5=0$ would follow from a calculation analogous to (\ref{longform}).
\end{proof}

Lemma \ref{intersect} implies $\lambda^6\ne 0$, finishing the proof of Proposition \ref{6sweepout}.

\medskip{\noindent \bf Step 4.} Finally, let us prove Lemma \ref{a0a1}.

\begin{proof}[Proof of Lemma \ref{a0a1}]
For Lemma \ref{a0a1} (3): $[A_3]=[A_0]$ because we can homotope $A_0$ to $A_3$ using the
$\{(1-s)a_0=sa_3\}$-bundles over $B$, for $0\leq s\leq 1$. (Note that we cannot prove, say, $[A_2]=[A_0]$ this way because the $\{(1-s)a_0=sa_2\}$-bundles over $B$ is not well-defined.)

We now  prove  Lemma \ref{a0a1} (1), by acting on the basic elements
$[c_1],[c_2],$ and $[c_3]$ of $H_1(X;\Z_2)$. More precisely, recall that $\lambda_i$ is by definition the Hom-dual of $c_i$. So by Lemma \ref{PD1} and \ref{PD2}, $PD([A_0]+[X|_{b_3}])$ acts on $[c_1],[c_2]$, and $[c_3]$ to give 1+0, 0+0, and 0+1 respectively.
Therefore, to prove (1) it suffices to show that
\begin{itemize}
    \item $PD([A_1])[c_1]=1$.
    \item $PD([A_1])[ c_2]=0$.
    \item $PD([A_1])[ c_3]=1$.
\end{itemize}

To show that $PD([A_1])[c_1]=1$, just observe that we can homotope $c_1$ to the loop $
    \{[0:a_1:0:0:a_4]:a_1^2+a_4^2=1\}
$ within the same fiber $X|_{D_2\cdot \id}$ as $\pi_1(\RP^4)=\Z_2$, and this loop intersects $A_1$ only once.

To show that $PD([A_1])[c_2]=0$, we will perturb $c_2$ to another loop $\tilde c_2$ as follows. Let $d_2:[0,1]\to SO(3)$ be the path that lifts $p\circ c_2\subset B=SO(3)/D_2$, starts at $\id$, and ends at $g_1$. Fix a small constant $\epsilon_0>0$. We define $\tilde c_2:[0,1]\to X$ to be such that it is over the same  base $p\circ c_2$ as $c_2$, but has different fibers:
\begin{align}\label{loop2}
    \tilde c_2(s):= &\;D_2\cdot ([1:\epsilon_0:-\epsilon_0:0:0],d_2(s))
\end{align}
Then one can check that $\tilde c_2(0)=\tilde c_2(1)$ so that $\tilde c_2$ is  a loop, and $\tilde c_2$ does not intersect $A_1$. Thus $PD([A_1])[c_2]=0$.

To show that $PD([A_1])[c_3]=1$, we will perturb $c_3$ to $\tilde c_3$ as follows. Let $d_3:[0,1]\to SO(3)$ be the path that lifts $p\circ c_3$, starts at $\id$, and ends at $g_2$. This time, we let $\epsilon$ be a function from $[0,1]$ to $\R$ that  strictly decreases from $\epsilon_0$ to $-\epsilon_0$, for some fixed small $\epsilon_0>0$. We define $\tilde c_3:[0,1]\to X$ by:
\begin{align}\label{loop3}
    \tilde c_3(s):=&\;D_2\cdot ([1:\epsilon(s):-\epsilon(s):0:0],d_3(s))
\end{align}
One can again check  $\tilde c_3$ is indeed a loop, but  $\tilde c_3$ intersects $A_1$ at one point: where $\epsilon(s)=0$. Thus $PD([A_1])[c_3]=1$.
This finishes the proof of Lemma \ref{a0a1} (1).

The proof of Lemma \ref{a0a1} (2) is  similar. We only state the modifications needed:
To prove $PD([A_2])[c_2]=1$ and $PD([A_2])[c_3]=0$, instead of (\ref{loop2}) and (\ref{loop3}), respectively, we use 
$$
    D_2 \cdot([1:\epsilon(s):\epsilon(s):0:0],d_2(s))\textrm{ and }
     D_2\cdot([1:\epsilon_0:\epsilon_0:0:0],d_3(s)).$$

The proof of Lemma \ref{a0a1} (4) is also similar. We only state the modifications needed: To prove $PD([A_4])[c_2]=1$ and $PD([A_4])[c_3]=1$, instead of (\ref{loop2}) and  (\ref{loop3}), respectively, we use 
$$
    D_2\cdot ([1:0:0:0:\epsilon(s)],d_2(s))
\textrm{ and }
    D_2\cdot([1:0:0:0:\epsilon(s)],d_3(s)).
$$

\end{proof}

\subsection{Proof of Proposition \ref{localmax}}\label{localmaxproof}$\;$

\medskip{\noindent \bf Step 1.}
For convenience, let us reparametrize the family $\Phi_5:\RP^4\x[0,1]\to\Zc$ as follows. First, we write $\RP^4$ as $\R^4\sqcup\RP^3$ in which $\RP^3$ is where $a_0=0$.  Then on $\R^4\x[0,1]$, we reparametrize the family  $\Phi_5$ by
$$\Phi_5(b_1,b_2,b_3,{b_4},b_5):=\{(x-b_1)^2-(y-b_2)^2+b_3z+{b_4}+b_5z^3=0\}\cap \mathbb B^3.$$
Throughout this section {\it we will adopt this new parametrization}. And then our goal is to show that $\area\circ\Phi_5$ has a strict local maximum at $(0,0,0,0,0)$. In fact, it suffices to prove:
\begin{prop} \label{neg0}
There exists $\epsilon_1,\epsilon_2>0$ such that for any $(b_1,b_2)\in (-\epsilon_1,\epsilon_1)^2$, $(b_3,b_4,b_5)\in \R^2\x [0,1]$ such that $b_3^2+b_4^2+b_5^2=1$, and $t\in (0,\epsilon_2)$, 
\begin{equation}\label{decre}
    \area(\Phi_5(b_1,b_2,b_3t,b_4t,b_5t))<\area(\Phi_5(0,0,0,0,0)).
\end{equation}
\end{prop}
Geometrically, $t$ governs how much the surface opens up (see Remark \ref{beh}). Namely, it is elementary to show that the width of the hole opened up in $\Sigma_t$ is at most $\sqrt{3t}$.

\medskip{\noindent \bf Step 2.}
We begin to prove Proposition \ref{neg0}. Let  $b_1,b_2,b_3,b_4,b_5$ satisfy the assumptions --- we will explain how small $\epsilon_1,\epsilon_2$ need to be later. For each $t\in(0,\epsilon_2)$, denote 
$\tilde\Sigma_t:=\Phi_5(b_1,b_2,b_3t,b_4t,b_5t).$ 
Then by Lemma \ref{firstvar},
$$\frac d{dt}\area(\tilde\Sigma_t)=-\int_{\tilde\Sigma_t}{\bf H}\cdot V-\int_{\partial\tilde\Sigma_t}\frac{{\bf n}\cdot w}{{\nu}\cdot w}\;V\cdot {\bf n},$$
which one would hope to show to be negative in order to prove Proposition \ref{localmax}.
However, the second integral is difficult to bound, since $\nu\cdot w$ can be zero on $\partial \mathbb B^3$.  {\it To prevent $\nu\cdot w=0$ on $\partial \mathbb B^3$, we will  slightly enlarge $\mathbb B^3$ to some domain $\Omega$}, which we will soon define. Then
for each $s\in (0,t]$, we let
\begin{equation}\label{sigma}
    \Sigma_s:=\{(x-b_1)^2-(y-b_2)^2+s(b_3z+{b_4}+b_5z^3)=0\}\cap \Omega.
\end{equation}
(Here $\Sigma_s$ has a boundary, so the notation is different from Lemma \ref{firstvar}.) Then
\begin{align}
    \nonumber &\;\area(\Phi_5(b_1,b_2,b_3t,b_4t,b_5t))-\area(\Phi_5(0,0,0,0,0))\\
\nonumber <&\;\area(\Sigma_t)-\area(\Phi_5(0,0,0,0,0))\\
\label{neg}=&\;\left(\area(\Sigma_0)-\area(\Phi_5(0,0,0,0,0))\right)+\int_0^t\frac d{ds}\area(\Sigma_s)ds
\end{align}
Therefore to prove Proposition \ref{neg0}, it suffices to show that expression (\ref{neg}) is negative. We will achieve this by showing {\it the initial area added by enlarging $\mathbb B^3$ to $\Omega$, which is the first term of (\ref{neg}), is dominated by the area decrease as $s$ increases from 0 to $t$, which is the second term of (\ref{neg}).}

But let us first define  $\Omega$. We now fix $t\in(0,\epsilon_2)$ also.  The new region $\Omega$ {\it will depend on} $t$ as follows. Let 
\begin{equation}\label{Rdef}
    R:=20\sqrt{t}.
\end{equation} Let $S_1,S_2,S_3$ be the {\it solid} cylinders in $\R^3$ with axis $\{x=b_1,y=b_2\}$ and radius $R,2R,\frac 14$ respectively. By letting $\epsilon_2$ be small we can assume $S_2\subset S_3$. Let $\Omega$ be the unit 3-ball with a bump within $S_2$, such that {\it $\partial \Omega$ becomes horizontal in $S_1$} (see Figure \ref{fig:omega}). Moreover, let us view $\partial\mathbb B^3\cap S_2$ (resp. 
$\partial\Omega\cap S_2$)  as the 2-sheeted graph of some function $\pm f$ (resp.  $\pm g$) over a disk $D$ on the $xy$-plane. Then since $\dist(p,(0,0))<2\epsilon_1+40\sqrt{\epsilon_2}$ for all $p\in D$, we know that $|\nabla f|\lesssim (2\epsilon_1+40\sqrt{\epsilon_2})^2$  (here $\lesssim$ means the inequality holds up to a multiplicative constant that is universal), thus {\it we can also assume   $|\nabla g|\lesssim (2\epsilon_1+40\sqrt{\epsilon_2})^2$}. As a result, 
if we write the outward unit normal $w$ of $\Omega$ as $(w_1,w_2,w_3)$, then for $\epsilon_1,\epsilon_2$ sufficiently small, we have in $S_2$
\begin{equation}\label{wcomp}
    |\nabla g|,|w_1|,|w_2|<C''(\epsilon_1+\epsilon_2)
\end{equation}
for some universal constant $C''$.

\begin{figure}
\centering
\includegraphics[width=3in]{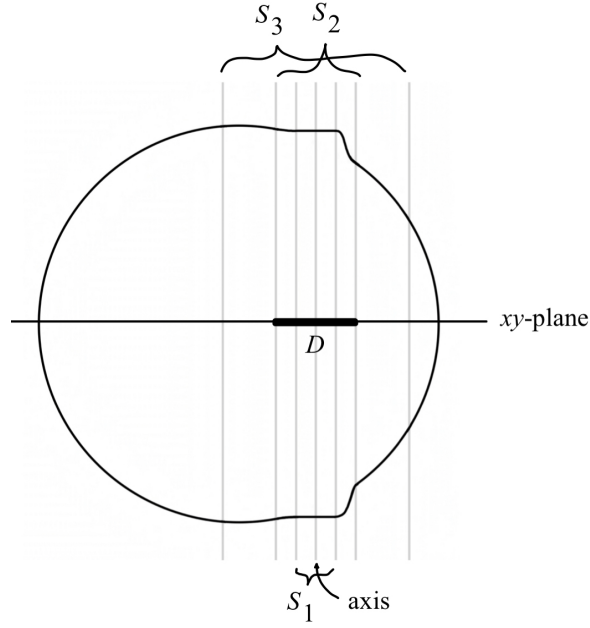}
\caption{A schematic picture of $\Omega$.}
\label{fig:omega}
\end{figure}

Now we have defined $\Omega$, {\it it suffices to show that expression (\ref{neg}) is negative.} The second term of (\ref{neg}) can be computed using the first variation formula in Lemma \ref{firstvar}. In order to estimate, let us derive some preliminary results in the next step.

\medskip{\noindent \bf Step 3.} We are interested in surfaces $\Sigma_s$ defined in (\ref{sigma}), which is the zero set in $\Omega$ of the polynomial
$$p(x,y,z):=(x-b_1)^2-(y-b_2)^2+b_3sz+{b_4}s+b_5sz^3,$$
for $s$ increases from 0 to $t$, where $t\in (0,\epsilon_2)$ is fixed.
Note that 
\begin{align*} 
\partial_s p=&\;b_3z+b_4+b_5 z^3,\; \nabla p=(2(x-b_1),-2(y-b_2),b_3s+3b_5sz^2),\\ \Hess(p)=&\begin{pmatrix}
2 & 0 & 0\\ 
0 & -2 & 0\\ 
0 & 0 & 6b_5sz
\end{pmatrix},\; \Delta p=6b_5sz.
\end{align*}
Moreover, denoting ${\bf H}=H{\bf n}$,
\begin{align}
\nonumber H=&\;\frac{\nabla p\;\Hess(p)\nabla p^T-|\nabla p|^2\Delta p}{|\nabla p|^3}\\
\nonumber=&\;\frac 1{|\nabla p|^3} [8(x-b_1)^2-8(y-b_2)^2+6b_5sz(b_3s+3b_5sz^2)^2\\
\nonumber&-(4(x-b_1)^2+4(y-b_2)^2+(b_3s+3b_5sz^2)^2)(6b_5sz)]\\
\label{mean}=&\;\frac 1{|\nabla p|^3}[-8s\partial_s p-24b_5sz((x-b_1)^2+(y-b_2)^2)],
\end{align}
in which in the third equality we  used $p=0$ on $\Sigma_s$. We can choose $\frac{\nabla p}{|\nabla p|}$ as the normal vector field $\bf n$, and $V:=-\frac{\partial_s p}{|\nabla p|^2}\nabla p$ as the deformation vector field of $\Sigma_s$ (because by differentiating $p(s,{\bf x}(s))=0$ with respect to $s$, one has 
$\partial_sp+\nabla p\cdot {\bf x}'=0$). As a result, by (\ref{mean}) and Lemma \ref{firstvar},
\begin{align}
\nonumber\frac d{ds}\area(\Sigma_s)&=-\int_{\Sigma_s}{\bf H}\cdot V-\int_{\partial\Sigma_s}\frac{{\bf n}\cdot w}{{\nu}\cdot w}\;V\cdot {\bf n}\\
\label{6terms}&=\int_{\Sigma_s}-\frac{8 s(\partial_s p)^2}{|\nabla p|^4}
-\frac{24 b_5sz\partial_s p((x-b_1)^2+(y-b_2)^2)}{|\nabla p|^4}\\
\nonumber&\;\;\;-\left(\int_{\partial\Sigma_s\cap S_1}+\int_{\partial\Sigma_s\cap (S_2\backslash S_1)}+\int_{\partial\Sigma_s\cap (S_3\backslash S_2)}+\int_{\partial\Sigma_s\cap (\R^3\backslash S_3)}\right)\frac{{\bf n}\cdot w}{{\nu}\cdot w}\;V\cdot {\bf n}.
\end{align}
We can write this as a {\it sum} of six integrals, which we will denote in the above order as $I_1,I_2,...,I_6$.

Now, remember that to prove Proposition \ref{neg0}, it suffices to show that expression (\ref{neg}) is negative. We claim that it  suffices to prove:
\begin{lem}\label{seven} There exist some large universal constant $C>0$ and small $\epsilon_1,\epsilon_2>0$ such that the following is true. For any $(b_1,b_2)\in (-\epsilon_1,\epsilon_1)^2$, $(b_3,b_4,b_5)\in \R^2\x [0,1]$ such that $b_3^2+b_4^2+b_5^2=1$, and $0<s<t<\epsilon_2$, we have:
\begin{itemize}
    \item $I_1<-\frac{1}{C\sqrt s}<0$.
    \item $|I_2|,|I_3|,|I_4|,|I_6|<C$.
    \item $|I_5|<-C\log(400s)$.
    \item $\area(\Sigma_0)-\area(\Phi_5(0,0,0,0,0))<Ct$.
\end{itemize}
\end{lem}
Indeed, from this lemma it follows that in (\ref{neg}), when $\epsilon_1,\epsilon_2$ are small, the dominating term is $\int^t_0I_1ds$, which is of order $\sqrt t$ and is negative. Thus the expression (\ref{neg}) is negative, as desired.

\medskip{\noindent \bf Step 4.} We now begin to prove Lemma \ref{seven}, by bounding the seven quantities listed one by one.

First, $I_1<0$ is clear, so we just need to lower bound $|I_1|$. By  Lemma \ref{cubicfun}, there exists a universal constant $h>0$ and an interval $[z_0,z_0+\frac 18]\subset[-\frac 12,\frac 12]$  on which $\partial_s p=b_3z+b_4+b_5z^3>h$ or $\partial_s p<-h$. 

Let us first tackle the  case $\partial_s p>h$. Namely, we have
\begin{align*}
|I_1|&= \int_{\Sigma_s} \frac{8s(\partial_s p)^2}{|\nabla p|^4}\geq\int_{\Sigma_s\cap\{z_0<z<z_0+\frac 18\}} \frac{8sh^2}{(8(x-b_1)^2+4s\partial_sp+(b_3+3b_5z^2)^2s^2)^2}\\
&\geq  \int_{\substack{\Sigma_s\cap\{z_0<z<z_0+\frac 18\}\\\cap\{-\frac 12<x<\frac 12\} }}
\frac{8sh^2}{(8(x-b_1)^2+13s)^2}.
 \end{align*}
Note that in the first inequality we rewrote $(y-b_2)^2$ in $|\nabla p|^4$ using  $p=0$, and used that $\partial_s p>h$ for $z\in[z_0,z_0+\frac 18]$; while the second inequality holds because $|\partial_s p|\leq 3$ and  $(b_3+3b_5z^2)^2s^2<s$ if $\epsilon_2$ and thus $s$ is small. Moreover, note that the domain of the last integral is a two-sheeted graph over the rectangle $[-\frac 12,\frac 12]\x[z_0,z_0+\frac 18]$ on the $xz$-plane, and clearly the graph has a larger area than the rectangle (see Figure \ref{fig:projsurf}). As a result,
\begin{align*}
|I_1|&>2\cdot\frac 18\int^{1/2}_{-1/2}  \frac{8sh^2}{(8(x-b_1)^2+13s)^2}dx
\gtrsim \frac{1}{\sqrt s}.\end{align*}
This finishes proving $I_1<-\frac{1}{C\sqrt s}$ for the case  $\partial_s p>h$. The second case $\partial_s p<-h$ is similar: One would integrate with respect to $y$ instead of $x$ in the last step.
\begin{figure}
\centering
\includegraphics[width=3in]{projsurf2.pdf}
\caption{}
\label{fig:projsurf}
\end{figure}

To bound $I_2$, we observe that
\begin{equation}\label{norm}
    |\nabla p|^2\geq 4(x-b_1)^2+4(y-b_2)^2\geq 4|(x-b_1)^2-(y-b_2)^2|=4s|\partial_s p|.
\end{equation}
So 
$$|I_2|\leq \int_{\Sigma_s}|24b_5z|
\frac{|s\partial_sp|}{|\nabla p|^2}
\frac{(x-b_1)^2+(y-b_2)^2}{|\nabla p|^2}
\leq \int_{\Sigma_s}24\cdot 1\cdot 1.$$
Now, since when $\epsilon_1,\epsilon_2$ are small enough, $\Sigma_s$ is close to $\Sigma_0$, which is two disks, we can assume $\area(\Sigma_s)<3\pi$. Hence, $|I_2|<24\cdot 3\pi$.

To bound $I_3$, note that 
\begin{equation}\label{integrand2}
    \frac{({\bf n}\cdot w)(V\cdot {\bf n})}{{\nu}\cdot w}=\frac{({\bf n}\cdot w)(V\cdot {\bf n})}{\sqrt{1-({\bf n}\cdot w)^2}}=\frac{({\bf n}\cdot w)(-\partial_s p)}{|\nabla p|\sqrt{1-({\bf n}\cdot w)^2}}
=\frac{(\nabla p\cdot w)(-\partial_sp)}{|\nabla p|\sqrt{|\nabla p|^2-(\nabla p\cdot w)^2}}.
\end{equation}
Now, inside $S_1$, $\partial\Omega$ is horizontal by definition  and so $w=\pm e_3$. Thus 
$$\abs{\frac{({\bf n}\cdot w)(V\cdot {\bf n})}{{\nu}\cdot w}}
\leq\frac{|\nabla p\cdot w||\partial_sp|}{|\nabla p|^2-(\nabla p\cdot w)^2}
\leq\frac{|s(b_3+3b_5z^2)||\partial_sp|}{4(x-b_1)^2+4(y-b_2)^2}
\leq\frac{4s|\partial_sp|}{4s|\partial_sp|}\leq 1,$$
in which the third inequality used (\ref{norm}). Now by Remark \ref{beh}, $\partial \Sigma_s\cap S_1$ is two  hyperbolas near respectively the north and the south pole (see Figure \ref{fig::sigmas}). Since the radius of $S_1$ is $R$, it is elementary to show that $\length(\partial \Sigma_s\cap S_1)<C'R$ for some universal constant $C'$.  Hence, using (\ref{Rdef}), $|I_3|<C'R<20C'\sqrt{\epsilon_2}$, which is less than some universal constant, assuming, say, $\epsilon_2<1$.
\begin{figure}
\centering
\includegraphics[width=3in]{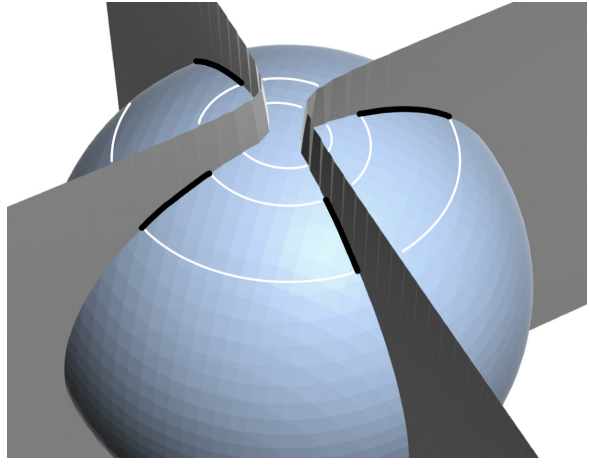}
\caption{This figure (not drawn to scale) shows $\partial \Omega$ intersecting the cubic surface defining $\Sigma_s$. The three white circles are $\partial\Omega\cap\partial S_i$, for $i=1,2,$ and $3$, which have radii $R,2R,$ and $\frac 14$ respectively. The black thick segments are $\partial\Sigma_s\cap( S_3\backslash S_2)$.}
\label{fig::sigmas}
\end{figure}

To bound $I_4$, using (\ref{wcomp}), we have
\begin{align}\label{nw}
    \abs{{\bf n}\cdot w}
    &\leq\frac{\abs{2(x-b_1)}}{|\nabla p|}\abs{w_1}+\frac{\abs{2(y-b_2)}}{|\nabla p|}\abs{w_2}+\frac{\abs{s(b_3+3b_5z^2)}}{|\nabla p|}\abs{w_3}\\
    \nonumber&\leq 1\cdot C''(\epsilon_1+\epsilon_2)+1\cdot C''(\epsilon_1+\epsilon_2) +\frac{4s}{\sqrt{4(x-b_1)^2+4(y-b_2)^2}}\cdot 1\\
    \nonumber&\leq 2C''(\epsilon_1+\epsilon_2)+\frac{4(R/20)^2}{2R}<\frac 1{10}. 
\end{align}
Note that in the second inequality we used that $s\leq t=(R/20)^2$ by (\ref{Rdef}), and in the last inequality assumed $\epsilon_1,\epsilon_2$ are small.

  Then using (\ref{integrand2}),
$$\abs{\frac{({\bf n}\cdot w)(V\cdot {\bf n})}{{\nu}\cdot w}}
=\frac{|{\bf n}\cdot w|{|\partial_s p|}}{|\nabla p|\sqrt{1-({\bf n}\cdot w)^2}}
\leq \frac{(1/10)\cdot 4}{2R\sqrt{1-(1/10)^2}}<\frac 1 R.$$
Again, it is elementary to show that the length of $\partial\Sigma_s\cap(S_2\backslash S_1)$ is less than $CR$ for some universal constant $C$. As a result, $|I_4|$ is less than $CR\cdot \frac 1 R=C$.

To bound $I_5$, using the fact that $p=0$ we can rewrite
\begin{align*}
\nabla p\cdot w&=2(x-b_1)x-2(y-b_2)y+s(b_3+3b_5z^2)z\\
    &=2(x-b_1)b_1-2(y-b_2)b_2+s(-b_3z-2b_4+b_5z^3).
\end{align*}
Therefore, using (\ref{Rdef}) and that $|\nabla p|>2R$, and assuming $\epsilon_1,\epsilon_2$ to be small,
\begin{align}\label{nw2}
    \abs{{\bf n}\cdot w}&\leq 
    \frac{\abs{2(x-b_1)}}{{|\nabla p|}}|b_1|+\frac{\abs{2(y-b_2)}}{|\nabla p|}|b_2|+\frac{\abs{s(-b_3z-2b_4+b_5z^3)}}{|\nabla p|}\\
    \nonumber&\leq 1\cdot \epsilon_1+1\cdot \epsilon_1+\frac{4s}{2R}\leq 2\epsilon_1+\frac{4(R/20)^2}{2R}<\frac 1{10}.
\end{align}
Then 
\begin{equation}\label{i5}
    \abs{\frac{({\bf n}\cdot w)(V\cdot {\bf n})}{{\nu}\cdot w}}
=\frac{|{\bf n}\cdot w||\partial_s p|}{|\nabla p|\sqrt{1-({\bf n}\cdot w)^2}}
\leq\frac{(1/10)\cdot 4}{\sqrt{4(x-b_1)^2}\sqrt{1-(1/10)^2}}
<\frac{1}{|x-b_1|}.
\end{equation}
Now, on $\Sigma_s\cap (S_3\backslash S_2)$, it follows from the definition that $|x-b_1|,|y-b_2|>R$. From this,  it is elementary to estimate ${\bf n}$ and show that in $S_3\backslash S_2$ the tangent planes of $\Sigma_s$ are close to that of $\{(x-b_1)^2-(y-b_2)^2=0\}$ (see Figure \ref{fig::sigmas}). In this step, the dependence of $R$ on $t$, namely $R=20\sqrt t$, is crucial: We choose $20\sqrt t$ because the width of the hole opened up in  $\Sigma_s$ is at most $\sqrt{3s}$, which we want to be  small compared to $R$. As a result, $\partial\Sigma_s\cap (S_3\backslash S_2)$ consists of eight  arcs such that if we let
$\rho$  be the orthogonal projection map from  $\partial\Sigma_s\cap (S_3\backslash S_2)$ to the $x$-axis, then the norm of the derivative $D\rho$ is lower bounded by some universal constant $C'>0$. In addition, the image $J$ of    $\rho$ is contained in $[b_1-\tfrac14,b_1-R]\cup[b_1+R,b_1+\tfrac14]$, and the preimage of each $x\in J$ has at most 4 points.
Therefore, by (\ref{i5}), for $\epsilon_2$ small, 
$$|I_5|\leq 4 \int_{[b_1-\tfrac14,b_1-R]\cup[b_1+R,b_1+\tfrac14]}\frac 1 {C'|x-b_1|} dx\lesssim -\log(R)\lesssim-\log(400s).$$

To bound $I_6$,
note that $\abs{\nabla p}>\frac 12$ outside $S_3$. Then as in (\ref{nw2}), we have ${\bf n}\cdot w<\frac 1 {10}$. It follows easily that the expression (\ref{integrand2}) is bounded by some universal constant. Then since $\length(\partial\Sigma_s\cap(\mathbb B^3\backslash S_3))$ bounded, so is $|I_6|$.

Finally, we prove the last item of Lemma \ref{seven}. Note that the difference between
$\Sigma_0$ and $\Phi_5(0,0,0,0,0)$ is $\{x^2-y^2=0\}\cap(\Omega\backslash \mathbb B^3)$, which is  four small planar pieces (two near the north pole and two near the south). Each piece can be contained in a rectangle of width $4R$ and, by (\ref{wcomp}), height $C''(\epsilon_1+\epsilon_2)(4R)$. As a result, using (\ref{Rdef}), $$\area(\Sigma_0)-\area(\Phi_5(0,0,0,0,0))\leq 4(4R)\cdot C''(\epsilon_1+\epsilon_2)(4R)\lesssim t.$$

This finishes the proof of Lemma \ref{seven}. Hence, we have proven Proposition \ref{neg0}, and thus Proposition \ref{localmax}.

\appendix
\section{Proof of Proposition \ref{globalmax}}\label{globalmaxproof}
The following proof is due to the MathOverflow user fedja \cite{fed}. Let $M$ denote the saddle  in $\R^3$ given by $x^2-y^2+z=0$. Then to prove Proposition \ref{globalmax},  it suffices, by rescaling, to show that for any ball $B$ with  center $(x_0,y_0,z_0)$ and radius $R>0$, the area of $M\cap B$ is less than $2\pi R^2$.

Recall that $M$ is foliated by straight lines: It can be parametrized by ${\bf x}(s,t)=(s+t,s-t,4st)$. 
Then  the Jacobian of $\bf x$ is  $2\sqrt{1+8s^2+8t^2}$. Thus we have
\begin{equation}\label{twoint}
    \area(M\cap B)<\iint_{\{(s,t):{\bf x}(s,t)\in B\}}(2\sqrt{1+8s^2}+2\sqrt{1+8t^2})dsdt.
\end{equation}
Now, for each fixed $s$, let $L_s$ be the corresponding coordinate line segment in $M\cap B$. Letting $d$ be the distance between $L_s$ and the center of $B$, we have
\begin{equation}\label{dist}
d^2\geq \min_{t\in\R}[(s+t-x_0)^2+(s-t-y_0)^2]= 2\left(s-\frac{x_0+y_0}2\right)^2.
\end{equation}
Note that $L_s$ is parameterized by a time interval of length
\begin{equation}\label{length}
    \frac{\length{(L_s)}}{\norm{\partial_t{\bf x}}}=\frac{2\sqrt{(R^2-d^2)^+}}{\sqrt{2(1+8s^2)}},
\end{equation}
where $^+$ denotes the positive part. It follows that, using (\ref{dist}) and (\ref{length}),
$$\iint_{\{(s,t):{\bf x}(s,t)\in B\}}2\sqrt{1+8s^2}dtds\leq\int_{s\in\R}2\sqrt 2\sqrt{\left[R^2-2\left(s-\frac{x_0+y_0}2\right)^2\right]^+} ds=\pi R^2.$$

The second integral  in (\ref{twoint}) can be similarly bounded, by integrating with respect to $s$ first. So  $\area(M\cap B)<2\pi R^2$, finishing the proof of Proposition \ref{globalmax}.

\section{First Variation Formula}
\begin{lem}\label{firstvar}
Let $\Omega$ be a compact $(n+1)$-dimensional region with smooth  boundary in $\R^{n+1}$, $\{\Sigma_s\}$  a 1-parameter family of hypersurfaces without boundary in $\mathbb R^3$, and $V$  a deformation  vector field of $\{\Sigma_s\}$. Then
$$\frac d{ds}\area(\Sigma_s\cap\Omega)=-\int_{\Sigma_s\cap\Omega}{\bf H}\cdot V-\int_{\partial(\Sigma_s\cap\Omega)}\frac{{\bf n}\cdot w}{{\nu}\cdot w}\;V\cdot {\bf n},$$
where $\bf n$ is a chosen unit normal vector field of $\Sigma_s$, $w$ the outward unit normal of $\partial \Omega$, and $\nu$ the outward unit conormal of $\Sigma_s$ on $\partial\Omega$.
\end{lem}
\begin{proof}
We first smoothly extend $w$ to a unit vector field on a neighborhood of $\partial\Omega$ in $\Omega$, and $\nu$ to a unit {\it tangent} vector field on a neighborhood of $\partial\Sigma_s$ in $\Sigma_s$. Let $\epsilon>0$, and $\Omega_\epsilon \subset \Omega $  be where the distance from $\partial \Omega$ is at least $\epsilon$.  Then by using the function $\dist(\cdot,\partial\Omega)$ on $\Omega$, with suitable smoothening, we can
approximate the indicator function  $\chi_\Omega$ by a smooth function $\chi^\epsilon_\Omega$ that is 0 outside $\Omega$ and 1 on $\Omega_\epsilon$, with $\nabla\chi^\epsilon_\Omega=-|\nabla\chi^\epsilon_\Omega|w$ in between.

Now, using the first variation formula (4.2) in \cite[p.49]{Eck12},
$$\frac d{ds}\int_{\Sigma_s}\chi^\epsilon_\Omega=\int_{\Sigma_s}-\chi^\epsilon_\Omega{\bf H}\cdot V+\nabla\chi^\epsilon_\Omega\cdot{\bf n}\; V\cdot{\bf n}.$$
Note that $$\nabla\chi_\Omega^\epsilon\cdot{\bf n}\; V\cdot{\bf n}=-|\nabla\chi_\Omega^\epsilon|\;w\cdot{\bf n}\; V\cdot{\bf n}=\nabla\chi_\Omega^\epsilon\cdot \nu\;\frac {w\cdot{\bf n}}{w\cdot \nu}\; V\cdot{\bf n}.$$
Denoting $g:=\frac {w\cdot{\bf n}}{w\cdot \nu}\; V\cdot{\bf n}$, we then have
$$\int_{\Sigma_s}\nabla\chi_\Omega^\epsilon\cdot{\bf n}\; V\cdot{\bf n}=\int_{\Sigma_s}\nabla\chi_\Omega^\epsilon\cdot g\nu =-\int_{\Sigma_s}\chi_\Omega^\epsilon\div(g\nu )\xrightarrow{\epsilon\to 0}-\int_{\Sigma_s\cap\Omega}\div(g\nu )=-\int_{\partial(\Sigma_s\cap\Omega)}g,$$
in which the second and the third equality are due to divergence theorem. Hence, $\frac d{ds}\area(\Sigma_s\cap\Omega)$ is equal to
$$
    \lim_{\epsilon\to 0}\frac d{ds}\int_{\Sigma_s}\chi^\epsilon_\Omega=\lim_{\epsilon\to 0}\int_{\Sigma_s}-\chi^\epsilon_\Omega{\bf H}\cdot V+\nabla\chi^\epsilon_\Omega\cdot{\bf n}\; V\cdot{\bf n}=-\int_{\Sigma_s\cap\Omega}{\bf H}\cdot V-\int_{\partial(\Sigma_s\cap\Omega)}g.
$$
\end{proof}

\section{A Lemma about Cubic Polynomials}
\begin{lem}\label{cubicfun}
There exists $h>0$ such that the following is true. For any $a,b,c$ such that $a^2+b^2+c^2=1$, define $f:[-\frac 12, \frac 12]\to\R$ by  $f(x)=ax^3+bx+c$. Then there exists some interval of length $\frac 18$ in $[-\frac 12,\frac 12]$ on which $|f|>h$. 
\end{lem}
\begin{proof}
Assume, by contradiction, that for each positive integer $n$ there exists a cubic function $f_n(x)=a_nx^3+b_nx+c_n$, with $a^2_n+b^2_n+c^2_n=1$ such that there is no interval of length $\frac 18$ in $[-\frac 12,\frac 12]$ on which $|f|>\frac 1n$. For each $n$, let $x_i$, for $i$ runs from 1 to at most 3, be the roots of $f_n(x)=0$, and $I_i\subset [-\frac 12,\frac 12]$ be the maximal interval containing $x_i$ on which $|f_n|<\frac 1n$. Then $[-\frac 12,\frac 12]\backslash(I_1\cup I_2\cup I_3)$  is a union of at most 4 intervals, each of which has length at most $\frac 18$. Thus, $I_1\cup I_2\cup I_3$ has length at least $1-4\cdot \frac 18=\frac 12$, and on it $|f_n|<\frac 1n$. Then it follows easily that $\sup_{x\in [-\frac 12,\frac 12]}|f'_n(x)|\to 0$ as $n\to\infty$. Since $f'_n(x)=3a_nx^2+b_n$, we must have $a_n\to 0$ and $b_n\to 0$ too, which forces $c_n\to 1$ since $a^2_n+b^2_n+c^2_n=1$. But then $f_n$ is very close to $1$ on $[-\frac 12,\frac 12]$, contradicting that $|f_n|<\frac 1n$ on a set of length at least $\frac 12$. 
\end{proof}

\bibliographystyle{alpha}
\bibliography{ref}

\end{document}